\documentclass[a4paper,10pt]{article}
\usepackage{amsmath,amssymb}
\usepackage[colorlinks,citecolor=blue]{hyperref}
\usepackage[utf8]{inputenc}
\usepackage[all]{xy}
\usepackage{graphicx}
\usepackage{euler}
\usepackage{color}

\newcommand{\QQ}{\mathbb{Q}}
\newcommand{\NN}{\mathbb{N}}
\newcommand{\ZZ}{\mathbb{Z}}

\newcommand{\PP}{\mathcal{P}}

\newcommand{\arb}[1]{\includegraphics[height=5mm]{a#1.pdf}}

\newcommand{\prelie}{\operatorname{PreLie}}
\newcommand{\lie}{\operatorname{Lie}}

\newcommand{\pl}{\curvearrowleft}

\newcommand{\pun}{\includegraphics[height=2.5mm]{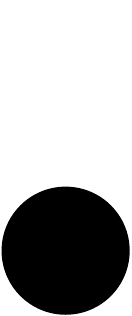}}

\newcommand{\corol}{\mathtt{Crl}}
\newcommand{\COR}{\textsc{Crls}}
\newcommand{\sumt}{\textsc{H}}
\newcommand{\linear}{\mathtt{Lnr}}
\newcommand{\fork}{\mathtt{Frk}}

\newcommand{\dend}{\operatorname{Dend}}

\newcommand{\dun}{\includegraphics[height=3mm]{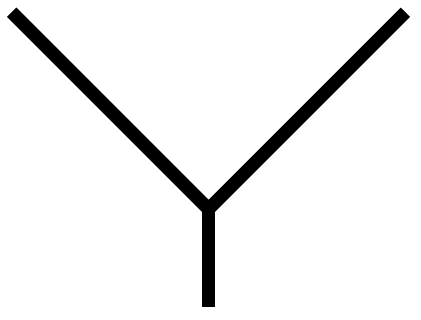}}

\newcommand{\Id}{\operatorname{Id}}
\newcommand{\gp}{\mathsf{G}}
\newcommand{\diam}{\diamond}
\newcommand{\aut}{\operatorname{aut}}

\newcommand{\compo}[2]{#1 \pmb{(}#2\pmb{)}}

\newcommand{\fl}{\mathbb{F}}

\newcommand{\valor}{\mathbb{V}}
\newcommand{\ca}{\operatorname{\mathbf{ca}}}

\renewcommand{\phi}{\varphi}
\newcommand{\eps}{\epsilon}
\newcommand{\prt}{\operatorname{Part}}
\newcommand{\rt}{\operatorname{rt}}

\newcommand{\sym}{\mathfrak{S}}

\newcommand{\sD}{\mathsf{D}} 
\newcommand{\sE}{\mathsf{E}} 
\newcommand{\sF}{\mathsf{F}} 
\newcommand{\sZ}{\mathsf{Z}} 
\newcommand{\sL}{\mathsf{L}} 
\newcommand{\sR}{\mathsf{R}} 
\newcommand{\sU}{\mathsf{U}} %
\newcommand{\sV}{\mathsf{V}} %

\newcommand{\xD}{\mathscr{D}} 
\newcommand{\xE}{\mathscr{E}} 
\newcommand{\xF}{\mathscr{F}} 
\newcommand{\xY}{\mathscr{Y}} 
\newcommand{\xZ}{\mathscr{Z}} 

\newcommand{\pp}{\includegraphics[height=3mm]{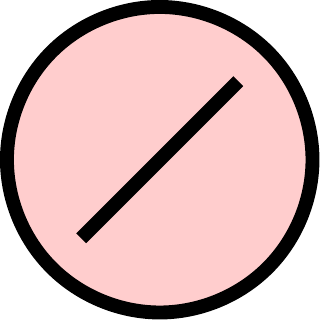}}
\newcommand{\mm}{\includegraphics[height=3mm]{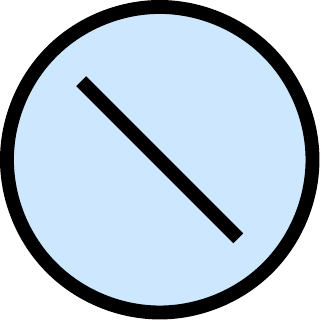}}
\newcommand{\dP}{\mathbf{P}} 
\newcommand{\dN}{\mathbf{N}} 
\newcommand{\dT}{\mathbf{T}} 
\newcommand{\dU}{\mathbf{U}} %
\newcommand{\dV}{\mathbf{V}} %
\newcommand{\ncsf}{\mathbf{Sym}}

\newcommand{\fqsym}{\mathbf{FQSym}}
\newcommand{\al}{\operatorname{Lie}}

\newtheorem{theorem}{Theorem}[section] 
\newtheorem{proposition}[theorem]{Proposition} 
\newtheorem{conjecture}[theorem]{Conjecture} 
\newtheorem{corollary}[theorem]{Corollary} 
\newtheorem{lemma}[theorem]{Lemma}

\newtheorem{remark}[theorem]{Remark}

\newtheorem{question}[theorem]{Question}

\newenvironment{proof}{\begin{trivlist}\item{\bf{Proof.}}}
  {\hfill\rule{2mm}{2mm}\end{trivlist}}

\title{Flows on rooted trees and \\the Narayana idempotents}
\author{F. Chapoton}
\date{\today}

\begin{document}

\maketitle

\begin{abstract}
  Several generating series for flows on rooted trees are introduced,
  as elements in the group of series associated with the Pre-Lie
  operad. By combinatorial arguments, one proves identities that
  characterise these series. One then gives a complete description of
  the image of these series in the group of series associated with the
  Dendriform operad. This allows to recover the Lie idempotents in the
  descent algebras recently introduced by Menous, Novelli and
  Thibon. Moreover, one defines new Lie idempotents and conjecture the
  existence of some others.
\end{abstract}


\section*{Introduction}

Let us start by introducing the context of this work, that can
summarized by the following diagram.

\begin{equation*}
\xymatrix{
 \ncsf \ar@{ ->}[r]       & \dend \ar@{ ->}[r]       & \fqsym       \\
 \ncsf \cap \al \ar@{ ->}[r]\ar[u] & \dend \cap \al \ar@{ ->}[r]\ar@{ ->}[u] &  \al\ar@{ ->}[u] \\
               & \prelie \ar@{ ->}[u]^{\phi}_{\simeq ?}      &
}
\end{equation*}

At the top left corner, $\ncsf$ is the graded Hopf algebra of
non-commutative symmetric functions \cite{ncsf}, which has a basis indexed by
compositions of integers. At the top right corner, $\fqsym$ is the
graded Hopf algebra of free quasi-symmetric functions, also known as
the Malvenuto-Reutenauer algebra \cite{malvenuto}, which has a basis
indexed by permutations. These two Hopf algebras can be considered as
non-commutative analogues of the classical Hopf algebra of symmetric
functions. They have been studied a lot, and have proved to be useful
in algebraic combinatorics, see for example \cite{thibon_lectures,ncsf6}.

At the middle of the top line, $\dend$ is the free Dendriform algebra
on one generator. This is also a graded Hopf algebra, also known as
the Loday-Ronco Hopf algebra \cite{loday_pbt}, and has a basis indexed
by planar binary trees. The horizontal morphisms of the first line are
inclusions of Hopf algebras, and can be described using appropriate
equivalence relations on permutations, see for instance
\cite{loday_pbt}.

On the second line, the subspace $\al$ of $\fqsym$ has two equivalent
descriptions. First, one can map $\fqsym$ into a space of rational
moulds, as described in \cite{moulds}. Then $\al$ is the subspace of
alternal elements, in the terminology of the mould calculus of Ecalle
\cite{ecalle_ari,ecalle_tale}. One can also identify $\fqsym$ with the
direct sum of all group rings of symmetric groups, and therefore to the
associative operad. Then $\al$ is the space of Lie elements, or the
image of the $\lie$ operad in the associative operad.

On the left of second line is the intersection of the subspaces
$\ncsf$ and $\al$ of $\fqsym$. It is known to be exactly the subspace
of primitive elements in the Hopf algebra $\ncsf$, by results of
\cite{ncsf}.

The intersection at the middle of the second line is quite
interesting. Starting from the usual injective morphism from the
Pre-Lie operad to the Dendriform operad, one gets an injective
morphism $\phi$ from the free Pre-Lie algebra on one generator,
denoted here by $\prelie$, to $\dend$. It was proved in \cite{moulds}
that its image is contained in the intersection $\dend \cap \al$.

It is conjectured that $\phi$ is an isomorphism from $\prelie$ to
$\dend \cap \al$. This has been checked for small degrees. If this
isomorphism holds, it would have interesting consequences for the
theory of Lie idempotents, that we will now present.

Recall that a Lie idempotent is an element $\theta$ in the group ring
$\QQ[\sym_n]$ of the symmetric group, such that $\theta$ is
idempotent, and such that the product by $\theta$ is a projector onto
the subspace of Lie elements. The set of Lie idempotents is an affine
subspace of the group ring $\QQ[\sym_n]$. There are many known
examples of Lie idempotents, and most of them belong to a sub-algebra
of $\QQ[\sym_n]$, the Solomon descent algebra.

There is a natural way to identify $\QQ[\sym_n]$ with the graded
component of degree $n$ of $\fqsym$. By this isomorphism, Solomon
descent algebra is identified with the graded component of degree $n$
of $\ncsf$. Moreover, the subspace of primitive elements of $\ncsf$
corresponds to the intersection of Solomon descent algebras with the
vector space spanned by Lie idempotents \cite{ncsf}.

From all this, one can deduce that any Lie idempotent in the descent
algebra gives an element in the intersection $\ncsf \cap \al$ and
therefore also in $\dend \cap \al$. If $\phi$ is an isomorphism, this
element will come from an element of $\prelie$. Conversely, given an
element of $\prelie$, if one can check that its image by $\phi$
belongs to $\ncsf$, then it will belong to $\ncsf \cap \al$ and
will define, up to multiplication by a scalar, a Lie idempotent in the
descent algebra.

Given any specific Lie idempotent in the descent algebra, one can
therefore ask for a description of its pre-image by $\phi$. This has
been obtained in \cite{qidempotent} for a one parameter familly of Lie
idempotents. The starting point of this article was to do the same for
a specific familly of Lie idempotents, that has just been recently
introduced. Let us now present them briefly.

Inspired by previous works by Ecalle and Menous
\cite{menous_bm,ecalle_menous} on the Alien calculus, Menous, Novelli
and Thibon have defined in \cite{menoth} a sequence of Lie idempotents
$\sD_n$ in the descent algebra of the symmetric group $\sym_n$. The
coefficients of $\sD_n$ in the basis of ribbon Schur functions are given by
homogeneous polynomials in two variables $a$ and $b$, more precisely
products of powers of $a$ and $b$ and Narayana polynomials in $a$ and
$b$. By homogeneity, one can let $a=1$ in the coefficients of $\sD_n$
without losing any information. We will therefore work with
polynomials in $b$ only.

By computing, for small $n$, the elements $\xD_n$ in $\prelie$ whose
image by $\phi$ is $\sD_n$, one observes that their coefficients are
positive polynomials in $b$ and seem to factorise according to
subtrees, with factors being also positive polynomials in $b$.

The first result of the present article is a combinatorial description
of the coefficients of $\xD_n$ and their factors, in terms of flows on
rooted trees.

To achieve this, one works inside groups of operadic series,
associated with the Pre-Lie and Dendriform operads. All the
idempotents $\sD_n$ are gathered into one series $\sD$ in the group
$\gp_{\dend}$ associated with the Dendriform operad. Their pre-images
$\xD_n$ by $\phi$ are similarly grouped in a series $\xD$ in the group
$\gp_{\prelie}$ associated with the Pre-Lie operad.

We proceed in the following order. First, we introduce the
combinatorial notion of flow on a rooted tree, and describe its
properties. Next, we obtain, from combinatorial arguments, various
functional equations satisfied by several series in the Pre-Lie group,
whose coefficients count different kinds of flows. We then go on to
introduce some series in the Dendriform group, and to show, by
algebraic means, that they satisfy another set of functional
equations. By comparing the functional equations in the Pre-Lie and
Dendriform cases, one can then recognize among the dendriform series
the images by $\phi$ of some of the Pre-Lie series.

On the way, one uses many auxiliary series, and some of them have
interesting properties. In particular, one does not only recover the
Lie idempotents $\sD_n$ of \cite{menoth}, but also gets a new familly
$\sF_n$ of Lie idempotents, related to closed connected
flows. Moreover, two other conjectural famillies $\sZ_n$ and
$\sF_{n,t}$ of Lie idempotents are proposed, for which we have not
been able to obtain a full proof. In the case of $\sF_{n,t}$, one is
missing a combinatorial proof of the existence of a Pre-Lie series
$\xF_t$ and so we do not know if the dendriform series $\sF_{t}$ is a
Lie element or not. In the case of $\sZ_n$, one only has a conjectural
description of the coefficients of the dendriform series $\sZ$, and so
we do not know if it belong to $\ncsf$.

\medskip

We gather in an appendix some technical tools that are necessary to
turn combinatorial bijections into equalities of series in groups
associated with operads. The notions of rooted-operad and
rooted-monoid that are introduced here may be of independent interest.

\section{Rooted trees and the $\prelie$ operad}

\subsection{Notations for rooted trees}

\label{section_pl}

A \textbf{rooted tree} is a finite connected and simply connected graph,
together with a distinguished vertex called the root.

Rooted trees will be considered implicitly as directed graphs by
orienting every edge towards the root. 

The \textbf{valency} $v_s$ of a vertex $s$ in a rooted tree is the
number of incoming edges.

The \textbf{height} of a vertex $s$ in a rooted tree is defined as
follows: the height of the root is $0$, and the height of the source
of every edge is $1$ more than the height of its end.

Rooted trees of maximal height at most $1$ are called \textbf{corollas}. Rooted
trees of maximal valency at most $1$ are called \textbf{linear trees}.

A rooted tree $T$ will sometimes be considered as a partially ordered
set whose Hasse diagram is given by the orientation towards the root, with
the root as the unique minimal element.

A \textbf{leaf} in a rooted tree $T$ is a vertex of valency $0$. A leaf can
also be defined as a maximal vertex.

Rooted trees will be drawn with their root at the bottom and leaves at the top.

If $T_1,\dots,T_k$ are rooted trees, we will denote
$B_+(T_1,\dots,T_k)$ the rooted tree obtained by grafting together
$T_1,\dots,T_k$ on a new common root.

Let $\arb{0}$ be the rooted tree with one vertex.

Let $\linear_\ell$ be the linear rooted tree with $\ell$ vertices,
defined by induction:
\begin{equation*}
  \linear_1=\arb{0}\quad\text{and}\quad \linear_{\ell+1}= B_+(\linear_\ell).
\end{equation*}

Let $\corol_n$ be the corolla with $n+1$ vertices, defined by
\begin{equation*}
  B_+(\arb{0},\dots,\arb{0}),
\end{equation*}
with $n$ copies of $\arb{0}$.

Let $\fork_{i,n-i}$ be the fork with $n$ vertices, with stem of size
$i$, defined by induction:
\begin{equation*}
  \fork_{1,\ell}=\corol_{\ell}\quad\text{and}\quad \fork_{k+1,\ell}=B_+(\fork_{k,\ell}).
\end{equation*}

Examples of linear trees, corollas and forks are depicted in figure
\ref{fig:fork}.

The number of vertices of a rooted tree $T$ will be denoted by $\#T$.

\begin{figure}
  \centering
  \includegraphics[height=3cm]{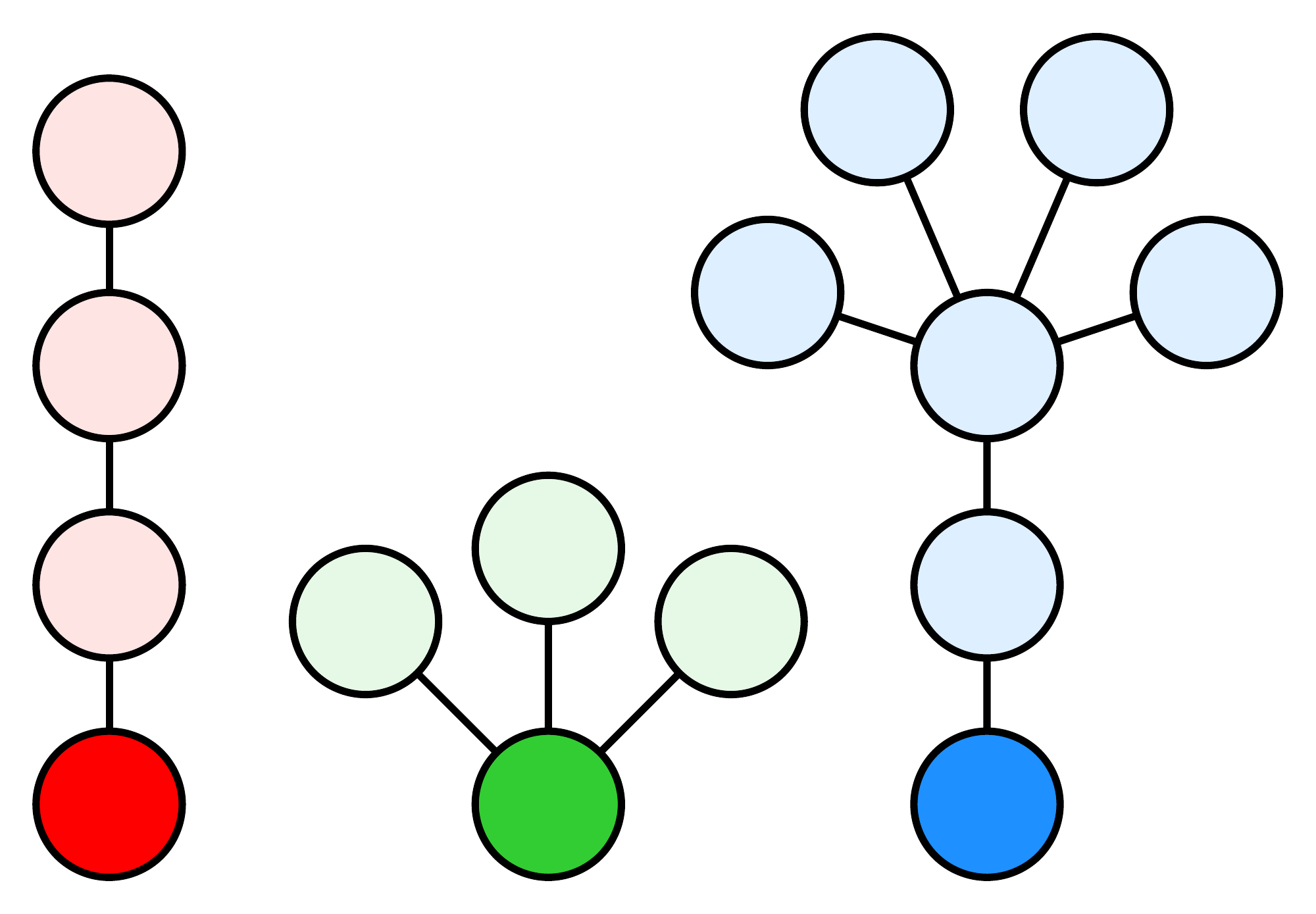}
  \caption{Linear tree $\linear_4$, corolla $\corol_3$ and fork $\fork_{3,4}$.}
  \label{fig:fork}
\end{figure}


\subsection{The group of rooted trees}

For more details on the general construction of the group of series
associated $\gp_\PP$ with an operad $\PP$, the reader may consult the
appendix \ref{appA}, \cite{chaplive2} and \cite[App. A]{qidempotent}.

We will work in the group of series $\gp_{\prelie}$ associated with
the Pre-Lie operad. This group is contained in the free Pre-Lie
algebra on one generator, denoted here by $\prelie$.

The Pre-Lie operad has a basis indexed by labelled rooted
trees \cite{chaplive1}. It follows that the Pre-Lie algebra on one
generator has a basis index by (unlabelled) rooted trees.

For a series $\xD$ in the group of rooted trees, we will use $\xD_T$
to denote the coefficient of the rooted tree $T$ in $\xD$, in the
following sense:
\begin{equation}
  \xD = \sum_{T} \frac{\xD_T}{\aut(T)} T,
\end{equation}
where $\aut(T)$ is the cardinal of the automorphism group of $T$.

The homogeneous component of $\xD$ of degree $n$ will be denoted by
$\xD_n$.

We will use the following special notation for the sum of all corollas:
\begin{equation}
  \label{defi_crls}
  \COR=\sum_{n \geq 0} \frac{\corol_n}{n!}.
\end{equation}


Let $\sumt_k$ be the element
\begin{equation}
  \sumt_k  =  \sum_{T} k^{\#T-1} \frac{T}{\aut_T}
\end{equation}
of the group $\gp_{\prelie}$. Its coefficients are polynomials in the
variable $k$.

\begin{lemma}
  \label{inversion_somme_tous}
  One has
  \begin{equation}
    \sumt_{k} \circ \sumt_{\ell} = \sumt_{k+\ell},
  \end{equation}
  where $k$ and $\ell$ are formal variables. In particular, when $k$
  is a positive integer, $\sumt_k$ is the $k^{th}$ power of $\sumt_1$ for the
  group law of $\gp_{\prelie}$. The inverse of $\sumt_k$ is $\sumt_{-k}$.
\end{lemma}
\begin{proof}
  It is enough to prove this identity for $k$ and $\ell$ positive
  integers, by polynomiality.

  Let $\mathsf{K}$ and $\mathsf{L}$ be finite sets of cardinality $k$
  and $\ell$. Elements of this sets are considered as colors.

  One applies proposition \ref{circ_prop} for the rooted-operad
  $\prelie$, with $A$ the species of rooted trees with edges colored
  by elements of $\mathsf{K}$, $B$ the species of rooted trees with
  edges colored by elements of $\mathsf{L}$ and $C$ the species of
  rooted trees with edges colored by elements of $\mathsf{K}\sqcup
  \mathsf{L}$. The series $s_A$, $s_B$ and $s_C$ are clearly just
  $\sumt_{k}$, $\sumt_{\ell}$ and $\sumt_{k+\ell}$.

  The necessary bijection (\textbf{hypothesis} $H_\sharp(A,B,C)$) is
  obtained as follows. Pick any rooted tree $T$ with edges colored by
  $\mathsf{K}\sqcup \mathsf{L}$. One considers the connected
  components in $T$ with respect to the edges with color in
  $\mathsf{L}$. Each connected component is a rooted tree. Collapsing
  every connected component to a point, one obtains a rooted tree
  $\tau$ with edges colored by $\mathsf{K}$. To recover the original
  rooted tree $T$, one has to know how to glue back the connected
  components into $\tau$. The different ways to do that are exactly
  counted by a constant of structure of the global composition map of
  the Pre-Lie operad.
\end{proof}

The suspension $\Sigma$ is defined by
\begin{equation}
  \Sigma \big{(} \sum_{n \geq 1}a_n \big{)} = \sum_{n \geq 1} (-1)^{n-1} a_n,
\end{equation}
where $a_n$ is homogeneous of degree $n$.

\section{Combinatorics of flows}

\subsection{Definition}

\begin{figure}
  \centering
  \includegraphics[height=4cm]{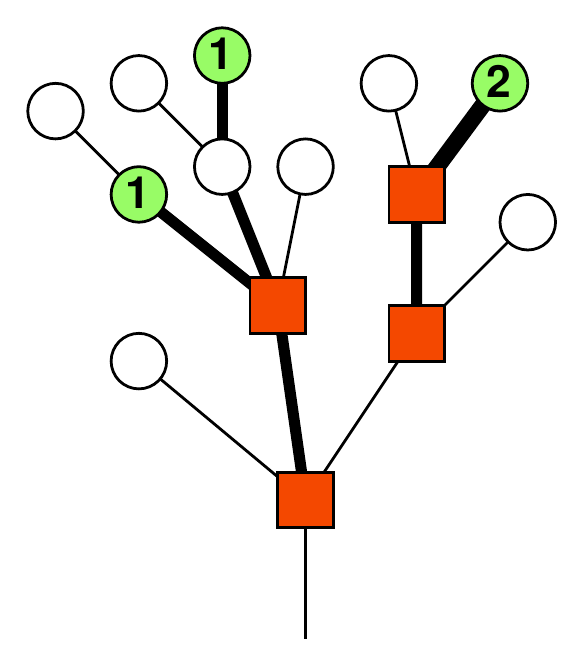}\includegraphics[height=4cm]{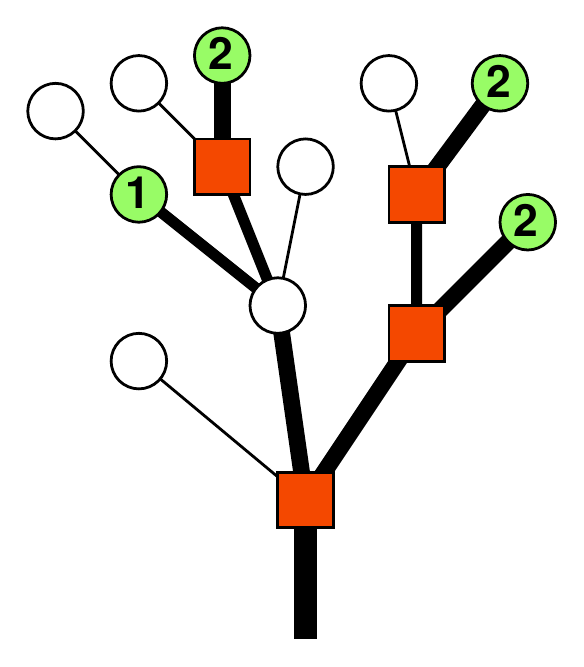}
  \caption{Two flows of size $4$, on the same rooted tree with $14$
    vertices. Only the left one is closed.}
  \label{fig:flow14}
\end{figure}

Let $T$ be a rooted tree. We will call a \textbf{flow} on $T$ of size
$k$ the data of 
\begin{itemize}
\item $k$ distinct vertices of $T$ (\textbf{outputs}),
\item vertices of $T$ (\textbf{inputs}), distinct from outputs,
  and that can be taken with multiplicities,
\end{itemize}
that has to satisfy the condition that we will introduce next.

Given inputs and $k$ outputs as above, one can define a \textbf{rate}
in $\ZZ$ on every edge of $T$ as follows.
\begin{itemize}
\item If the vertex $v$ is neither an input nor an output, the sum of
incoming rates in $v$ is equal to the outgoing rate of $v$.
\item If the vertex $v$ is an input with multiplicity $\ell$, the
  outgoing rate of $v$ is the sum of incoming rates in $v$ plus
  $\ell$.
\item If the vertex $v$ is an output, the
  outgoing rate of $v$ is the sum of incoming rates in $v$ minus $1$.
\end{itemize}
The main requirement is that \textit{all rates are in $\NN$}.

Note that, by convention, the incoming rate in leaves is $0$, but the
outgoing rate at the root (\textbf{exit rate}) can be an arbitrary
positive integer.

If the exit rate is $0$, the flow is \textbf{closed}.

This definition is illustrated in Figure \ref{fig:flow14}, where
outputs are depicted by red squares {\color{red}\rule{2mm}{2mm}} and
inputs by green circles {\color{green} $\bullet$} with their
multiplicity. The rates, between $0$ and $3$, are drawn with
increasing width.

\begin{lemma}
  A closed flow of size $k$ can also be described as
  \begin{itemize}
  \item $k$ distinct vertices of $T$ (outputs),
  \item $k$ vertices of $T$ (inputs), distinct from outputs, and that
    can be taken with multiplicities,
  \end{itemize}
  such that there exists $k$ decreasing paths from one input to an
  output that make a one-to-one matching of inputs with outputs.
\end{lemma}
\begin{proof}
  Let us see why the data of a closed flow is equivalent to the
  existence of $k$ paths with the required properties.

  Given $k$ decreasing paths matching inputs with outputs, one can
  find the rate of an edge by counting how many paths go through this
  edge. This rate function on edges does satisfy all the desired
  properties, and defines a closed flow.

  Conversely, given a rate function on edges defining a closed flow,
  one can find paths, by induction on the size $k$. Let us pick an
  output and choose an increasing path of edges of strictly positive
  rate, until one reaches an input. This defines a path from the
  reached input to the chosen output. Removing this input and this
  output and subtracting $1$ to the rate function for every edge of
  this path, one find another admissible rate function with $k$
  decreased by $1$. Then by induction, one gets $k$ paths with the
  expected matching property.
\end{proof}

Let $\fl(T)$ be the set of flows on $T$ and $\fl(T,k,i)$ be the finite
set of flows of size $k\in \NN$ with exit rate $i\in\NN$.

\subsection{Properties of closed flows}

Let us give some simple properties of the definition of closed flows.

For every rooted tree $T$, there is exactly one closed flow of size
$0$, which is the empty flow, with no input vertex and no output
vertex, where every edge has rate $0$.

For a rooted tree $T$, closed flows of size $1$ are in bijection with
pairs of distinct comparable vertices of $T$. The number of closed
flows of size $1$ is therefore the sum of the heights of the vertices
of $T$.

\begin{lemma}
  \label{lemma_maxi}
  For a rooted tree $T$, the maximal size of a flow on $T$ is the
  number of non-leaf vertices of $T$. There always exist a closed flow
  having this exact size.
\end{lemma}

\begin{proof}
  Indeed, any output must be a non-leaf vertex, because it has to be
  smaller than an input vertex. Conversely, one can find a closed flow
  of this size by putting an output on every non-leaf vertex and
  inputs on leaves as follows. Going upwards in the tree, one can
  choose at each output where the incoming flow should come from,
  until one reaches leaves.
\end{proof}

\subsection{Small flows}

Let us say that a flow $\psi \in \fl(T)$ is \textbf{small} if the root
is neither an output nor an input.

If the flow is closed, it is equivalent to require that the rate of
every edge incoming in the root of $T$ is $0$.

Denote by $\fl^{s}(T)$ the set of small flows on $T$.

\begin{lemma}
  \label{relation_produit_D_E}
  If $T=B_+(T_1,\dots,T_k)$, there is a bijection
  \begin{equation}
    \fl^{s}(T) \simeq \prod_{i=1}^{k} \fl(T_i),
  \end{equation}
  where the factors are given by restriction of the flow to subtrees.
\end{lemma}

\subsection{Inductive description of flows}

Let $\xE_{T,t}$ be the generating function of flows on $T$ with respect to size and exit rate:
\begin{equation}
  \xE_{T,t} = \sum_{k,i \geq 0} \sum_{\psi \in \fl(T,k,i)} b^{k} t^i,
\end{equation}
and let $\xD_{T,t}$ be the similar generating function of small flows on $T$:
\begin{equation}
  \xD_{T,t} = \sum_{k,i \geq 0} \sum_{\psi \in \fl^{s}(T,k,i)} b^{k} t^i.
\end{equation}

Recall that Lemma \ref{lemma_maxi} says in particular that the size of
a flow on $T$ is bounded by the number of non-leaf vertices of
$T$. Therefore the generating functions $\xE_{T,t}$ and $\xD_{T,t}$ are
polynomials in $b$ with coefficients that are formal power series in
$t$. We will see later that they are in fact polynomials in $b$ with
coefficients that are rational functions in $t$.

By Lemma \ref{lemma_maxi}, the degree of $\xE_T$ as a polynomial in
$b$ is exactly the number of non-leaf vertices of $T$. The constant
term of $\xE_T$ with respect to $b$ is ${1}/{(1-t)}^{\# T}$, because a
flow without outputs is just the choice of how many inputs there are
at every vertex.

For example, when $T$ is the fork $\fork_{2,2}$, one gets
\begin{equation*}
  \xE_{T,t}=\frac{ 1+ 5 b + 3 b^{2} - t (9 b + 8 b^{2}  )+ t^2 ( 5 b + 7 b^{2} ) - t^3 (b + 2 b^{2}  ) }{(1-t)^4}.
\end{equation*}

We will see later how to compute this by induction.

We will use the general convention that the value at $t=0$ of a series
denoted by a symbol with index $t$ will be denoted by the same symbol
without index $t$. For instance, let $\xE_{T}$ and $\xD_{T}$ be the
value at $t=0$ of $\xE_{T,t}$ and $\xD_{T,t}$.

\begin{lemma}
  \label{lemme_d_exp_e}
  One has $\xD_{B_+(T_1,\dots,T_k),t}=\prod_{i=1}^k\xE_{T_i,t}$.
\end{lemma}
\begin{proof}
  This follows from the bijection of Lemma \ref{relation_produit_D_E},
  and its simple behaviour with respect to size and exit rate.
\end{proof}

We will now proceed to give an inductive description of the series
$\xE_{T,t}$ and $\xD_{T,t}$.

Let $T$ be a tree and $v \to u$ be an edge of $T$, with $u$ closer to
the root. Let $T\pl_v w$ be the tree obtained by adding a new vertex
$w$ on top of $v$. Let $T\pl_u w$ be the tree obtained by adding a new
vertex $w$ on top of $u$. Let $S$ and $T_1,\dots,T_k$ be the trees
obtained from $T$ by removing the edges incoming in $v$. Here $S$ is
the bottom tree (containing the root of $T$) and $T_1,\dots,T_k$ are
the top trees. This is illustrated in figure \ref{fig:pullup}.

\begin{figure}
  \centering
  \includegraphics[height=4cm]{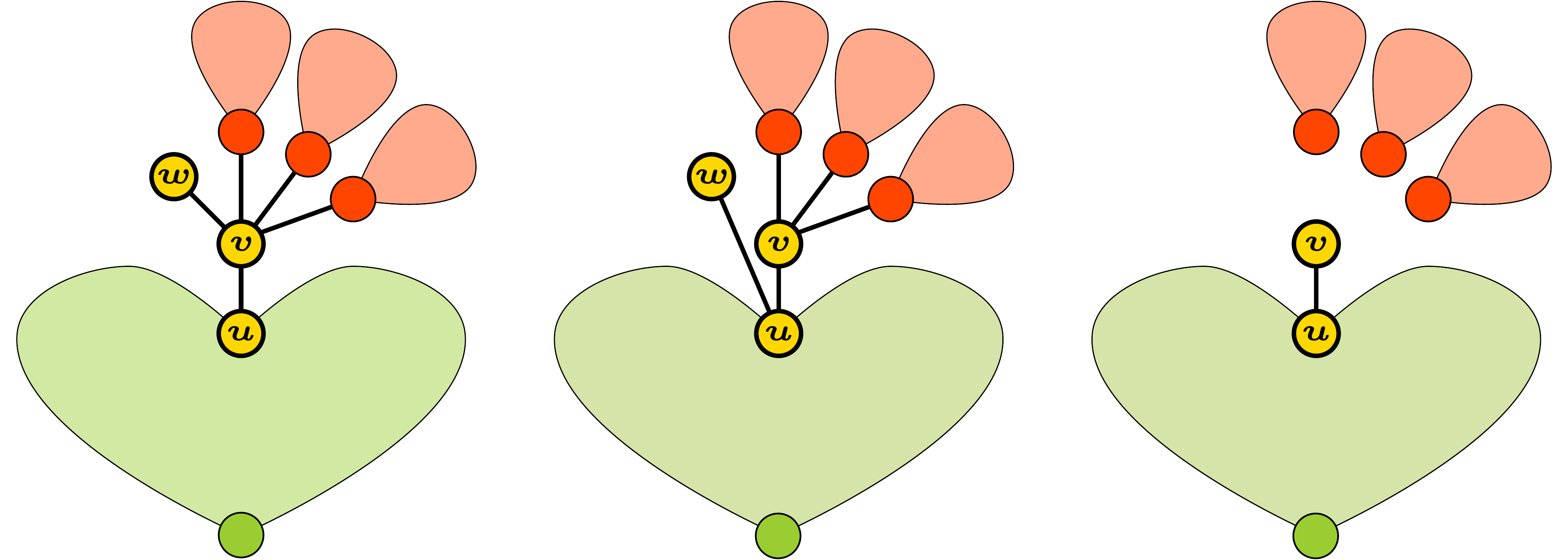}
  \caption{From left to right: $T\pl_v w$, $T\pl_u w$ and $S$ under $T_1,\dots,T_k$.}
  \label{fig:pullup}
\end{figure}

\begin{theorem}
  \label{main}
  With the previous notations, one has the following equalities:
  \begin{equation}
    \label{pullE}
    \xE_{T\pl_v w,t} = \xE_{T\pl_u w,t} + b\, \xE_{S,t} \prod_{i=1}^{k} \xE_{T_i},
  \end{equation}
  and
  \begin{equation}
    \label{pullD}
    \xD_{T\pl_v w,t} = \xD_{T\pl_u w,t} + b \,\xD_{S,t} \prod_{i=1}^{k} \xE_{T_i}.
  \end{equation}
\end{theorem}

\begin{proof}
  Let us prove the first equation.

  Let us consider a flow on the tree $T\pl_v w$. Let $\alpha$ be the
  rate of $w \rightarrow v$ and $\beta$ be the rate of $v \rightarrow
  u$. One can distinguish two cases.

  Either $\alpha=\beta+1$, in which case $v$ is an output, and all
  other edges incoming in $v$ have rate $0$. This kind of flow can be
  described in a bijective way using closed flows on the trees
  $T_1,\dots,T_k$ and one flow on the tree $S$. This gives the
  rightmost term.

  Otherwise $\alpha \leq \beta$. One can then define a flow on $T\pl_u
  w$ as follows. One moves down the end of the edge $w \to v$ which
  becomes an edge $w \to u$ and keep the rate $\alpha$. The rate of
  the edge $u-v$ is set to $\beta-\alpha$ and remains positive. This
  clearly defines a bijection, and one gets the leftmost term.

  Requiring in addition that the root is empty, the same proof
  gives the second identity.
\end{proof}

The simplest case of this induction is when $v$ is a leaf in $T$, in
which case the rightmost term has just the factor associated with $S$.

This theorem can be used to compute $\xE_{T,t}$ from smaller cases, by
choosing a leaf $w$ of height at least $2$. This is always possible,
unless $T$ is a corolla.

There is a nice commuting property to this induction. Indeed, one can
use it in several different ways to compute $\xE_{T,t}$, by choosing
different leaves. This happens first for trees with $5$ vertices.

One has the following consequence:
\begin{corollary}
  \label{plus_large}
  Let $T$ be a rooted tree. Then one has
  \begin{equation}
    \xE_{B_+(\pun,T_1,\dots,T_k),t} = \frac{1}{1-t}\left(\xE_{B_+(T_1,\dots,T_k),t} + b\, \prod_{i=1}^k \xE_{T_i}\right).
  \end{equation}
\end{corollary}
\begin{proof}
  This follows from equation \eqref{pullD}. Indeed, one has
  \begin{equation*}
    \xD_{B_+(B_+(\pun,T_1,\dots,T_k)),t}=\xD_{B_+(\pun,B_+(T_1,\dots,T_k)),t}+b \xD_{B_+(\pun),t} \prod_{i=1}^k \xE_{T_i}.
  \end{equation*}
  One can then use Lemma \ref{lemme_d_exp_e}.
\end{proof}

Corollary \ref{plus_large} can be used to compute the coefficients
$\xE_{\corol_n,t}$ for corollas, by induction on $n$.

\begin{remark}
  When $t=0$, Theorem \ref{main} implies that the coefficients of
  $\xE_{T}$ (as a polynomial in $b$) grow when a leaf is pulled up, as
  the rightmost term of \eqref{pullE} has positive coefficients.
\end{remark}

\subsection{Properties of $\xE_T$}

\begin{lemma}
  \label{rational}
  For every rooted tree $T$, the series $\xE_{T,t}$ is a polynomial in
  $b$ of degree the number of non-leaf vertices of $T$, with
  coefficients that are rational functions in $t$, with poles only at
  $t=1$. The common denominator of $\xE_{T,t}$ is $(1-t)^{\# T}$.
\end{lemma}
\begin{proof}
  The polynomial behaviour with respect to $b$ follows from the upper
  bound on the number of outputs, given by the number of non-leaf
  vertices, see Lemma \ref{lemma_maxi}. There always exists at least
  one flow with outputs at every non-leaf vertex, for example by
  placing sufficiently many inputs in every leaf. Therefore the degree
  of the polynomial is the number of non-leaf vertices.

  Let us now show that the coefficients of this polynomial in $b$ are
  rational functions in $t$ with poles only at $t=1$ and of order at
  most the size of $T$. This is true for the rooted tree $\arb{0}$, as
  $\xE_{\pun,t}=1/(1-t)$. By corollary \ref{plus_large}, this is true
  for all corollas, by induction. One can then use induction on the
  sum of heights of the vertices and on the number of vertices. Let
  $T$ be a tree which is not a corolla, and let $w$ be a leaf of
  maximal height in $T$. Take $v$ to be the vertex under $w$ and $u$
  the vertex under $v$. Then one can apply Theorem \ref{main} to prove
  the induction step.

  It remains to show that the order of the pole at $1$ of $\xE_{T,t}$
  is exactly the size of $T$. This follows from the obvious fact that
  the constant term with respect to $b$ is exactly $1/(1-t)^{\# T}$.
\end{proof}

The same kind of properties holds for $\xD_{T,t}$, thanks to Lemma
\ref{lemme_d_exp_e} and the obvious initial conditions
$\xD_{{\corol_n},t}=1/(1-t)^{n}$.

\subsection{Connected flows}

Let us say that two vertices $u,v$ of $T$ are connected by the flow
$\psi$ on $T$ if every edge of the unique path from $u$ to $v$ does
have a strictly positive rate in $\psi$.

One can then define connected components with respect to the flow
$\psi$, namely sets of vertices connected by the flow $\psi$. Each
connected component with respect to a flow is a rooted tree.

A flow is called \textbf{connected} if it has exactly one connected component.

Let $\fl^{c}(T)$ be the set of connected flows on $T$.

\begin{lemma}
  \label{connexe_un}
  If a rooted tree $T$ admits a closed connected flow, its root has
  valency at most $1$.
\end{lemma}
\begin{proof}
  The statement holds for the tree with one vertex. One can therefore
  assume that the tree $T$ is not the trivial tree $\arb{0}$. By connectedness,
  every edge incident to the root contributes at least $1$ to the
  total rate entering the root. By closure, the root is then
  necessarily an output, and it can only accept a rate of
  $1$. Therefore there is exactly one incident edge to the root.
\end{proof}

We will consider now the question of what rooted trees admit a
closed connected flow.

\subsection{Trees with a closed connected flow}

We will now give a description of the rooted trees that admit a closed
connected flow, using a function defined by Jean-Claude Arditti
\cite{arditti,arditti_cori} in relation to rooted trees with
Hamiltonian comparability graphs. One can note that these references
also use some kind of flows on rooted trees. To avoid possible
confusion, we will call this function the valor, which is not the
original terminology.

Let $T$ be a rooted tree. The \textbf{valor} $\valor(f)$ of a leaf $f$
is $1$. The valor $\valor(v)$ of a vertex $v$ is
\begin{equation}
  \max(1,-1+\sum_{s\to v} \valor(s)).
\end{equation}

\begin{lemma}
  The valor of the root of $T$ is the minimal value of the exit rate
  among all connected flows on $T$ with non-zero exit rate.
\end{lemma}
\begin{proof}
  By induction on the size of the tree $T$. This is true for the tree
  $\arb{0}$, which has minimal non-zero exit rate $1$. Let
  $T=B_+(T_1,\dots,T_k)$. Then the minimal exit rate of a connected
  flow on $T$ is the sum of the minimal non-zero exit rates of
  $T_1,\dots,T_k$, minus $1$ corresponding to an output at the root of
  $T$. If this is at least $1$, this is the minimum non-zero exit
  rate. If this is zero, the minimum non-zero exit rate is $1$, and
  can be obtained by adding $1$ to the rate along the path from the
  root to any chosen leaf.

  This proves that the minimal non-zero exit rate satisfies the same
  recursion as the valor.
\end{proof}

\begin{proposition}
  A rooted tree $B_+(T)$ admits a closed connected flow if and only if
  the root of $T$ has valor $1$.
\end{proposition}
\begin{proof}
  Using Lemma \ref{connexe_un}, the rooted tree $B_+(T)$ admits a
  closed connected flow if and only if the rooted tree $T$ admits a
  connected flow with exit rate $1$. By the previous lemma, this is
  equivalent to say that the valor of the root of $T$ is $1$.
\end{proof}

\subsection{Inductive description of connected flows}

Let us introduce a generating function for connected flows:
\begin{equation}
  \xE^c_{T,t} = \sum_{k,i \geq 0} \sum_{\psi \in \fl^{c}(T,k,i)} b^{k} t^i.
\end{equation}

By Lemma \ref{connexe_un}, a rooted tree (different from $\arb{0}$)
which admits a closed connected flow can be written $B_+(T)$. Let us
denote by $\xF_T$ the generating series of connected flows on $T$ with
exit rate $1$.

We will now obtain an inductive description of the coefficients $\xF_T$.

Let us consider the situation depicted in figure \ref{fig:pullupF},
with the same notations as for Theorem \ref{main}. The tree $S$ is
obtained from $T$ by removing everything above $v$. The tree $S'$ is
obtained as the subtree of $T\pl_v w$ with root $v$.

\begin{figure}
  \centering
  \includegraphics[height=4cm]{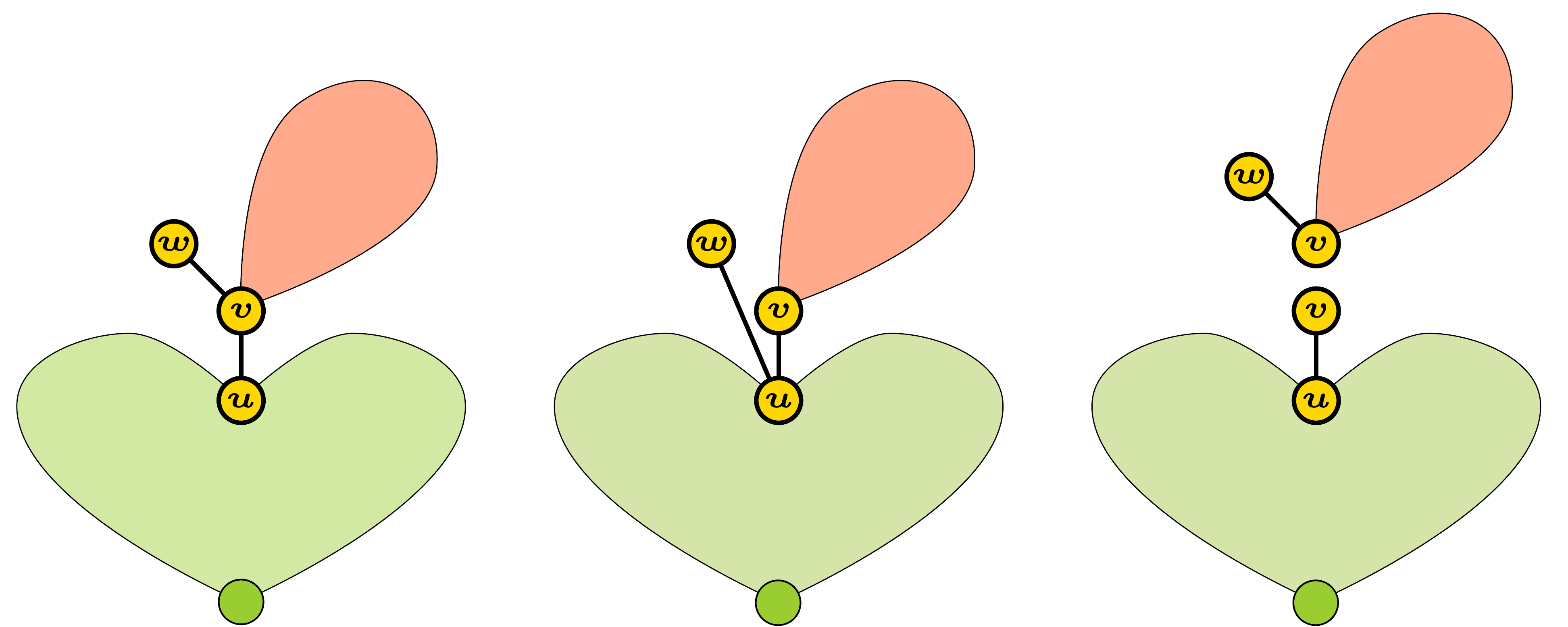}
  \caption{From left to right: $T\pl_v w$, $T\pl_u w$ and $S$ under $S'$.}
  \label{fig:pullupF}
\end{figure}

\begin{theorem}
  \label{mainF}
  With the previous notations, one has the following equalities:
  \begin{equation}
    \label{pullF}
    \xF_{T\pl_v w} = \xF_{T\pl_u w} + \xF_{S}\xF_{S'}.
  \end{equation}
\end{theorem}

\begin{proof}
  Let us consider a connected flow on $T \pl_v w$ with exit rate $1$.
  Let $\alpha\geq 1$ be the rate of the edge $w \to v$ and $\beta\geq
  1$ be the rate of the edge $v \to u$.

  If $\beta \geq \alpha+1$, then one can define a connected flow on $T
  \pl_u w$ with exit rate $1$ as follows. One replaces the edge $w\to
  v$ by an edge $w\to u$ with rate $\alpha$, and assign the rate
  $\beta-\alpha \geq 1$ to the edge $v \to u$. This is clearly a
  bijection, and gives the leftmost term.

  Otherwise, one has $\beta \leq \alpha$. One can then define a
  connected flow on $S'$ with exit rate $1$ and a connected flow on
  $S$ with exit rate $1$, as follows. On the bottom tree $S$, the
  vertex $v$ becomes an input with exit rate $\beta$, and all rates
  are unchanged. On the top tree $S'$, the vertex $v$ has the same
  content as the vertex $v$ of $T\pl_v w$, either input or output. One
  assigns to the edge $w \to v$ the rate $\alpha-\beta+1\geq 1$. One
  can check that the exit rate of this connected flow on $S'$ is
  $1$. This construction is clearly a bijection, and one obtains the
  rightmost term.
\end{proof}

For example, one can compute using this theorem that
$\xF_{\fork_{2,2}}$ is $2b(1+b)$.

\begin{corollary}
  For every rooted tree $T$ with $n$ vertices, the coefficient of
  $b^k$ and the coefficient of $b^{n-1-k}$ in $\xF_T$ are equal.
\end{corollary}
\begin{proof}
  This is certainly true for small corollas by inspection, and
  $\xF_{\corol_n}$ vanishes if $n\geq 3$. Then one can proceed by
  induction on the size and the total height, using \eqref{pullF}.
\end{proof}

One may wonder whether this unexpected symmetry has a combinatorial
description.

\medskip

It appears that it may be possible to introduce a parameter $t$ in the
inductive definition \eqref{pullF}.

\begin{conjecture}
  \label{conjecture_F}
  We keep the same notations as for Theorem \ref{mainF}. There exists
  rational functions $\xF_{T,t}$, such that
  \begin{equation}
    \xF_{T\pl_v w,t} = \xF_{T\pl_u w,t} + (1-t) \xF_{S,t}\xF_{S',t},
  \end{equation}
  and such that 
  \begin{equation}
    \xF_{\linear(n),t}=\xE_{\linear(n),t} \quad \text{for}\quad n\geq 1,
  \end{equation}
  and
  \begin{equation}
    \xF_{\corol_n}=b(-t)^{n-2}/(1-t)^{n-1} \quad \text{for}\quad  n\geq 2.
  \end{equation}
  
\end{conjecture}

It is easy to prove that this defines uniquely the fractions
$\xF_{T,t}$, if they exist.

For example, one gets that
\begin{equation*}
  \xF_{\fork_{2,2},t} = \frac{b}{(1-t)^2}+ \frac{b(1+2b)}{1-t}.
\end{equation*}

Looking at the first fractions $\xF_{T,t}$, one observes that they do
not have positive coefficients as formal power series in $t$ and $b$,
for example for the rooted tree $B_+(\corol_2,\arb{0},\arb{0})$. They
can therefore not be given a combinatorial description similar to the
one for $\xF_T$ in terms of connected flows with exit rate $1$.

\subsection{Flows on linear trees and Dyck paths}

\label{dyck}

Let us consider the case of the linear trees. We first show that
closed flows on linear trees are in bijection with very classical
objects, namely Dyck paths.

Recall that a \textbf{Dyck path} of length $2n$ is a plane lattice
path from $(0,0)$ to $(n,n)$ using steps $(0,1)$ (up) and $(1,0)$
(right) and keeping above the diagonal line $y=x$. A Dyck path of
length at least $2$ is called \textbf{indecomposable} if it only
touches the diagonal line at its extremities. Every Dyck path can by
uniquely written as the concatenation of indecomposable Dyck
paths. Every indecomposable Dyck path can be uniquely written
$(0,1)D(1,0)$ where $D$ is a Dyck path. A \textbf{peak} in a Dyck path
is a factor $(0,1)(1,0)$. We say that two letters $(0,1)$ and $(1,0)$
appearing in this order in a Dyck path are \textbf{matched} if the
factor between them is a Dyck path.

\begin{figure}
  \centering
  \includegraphics[height=3cm]{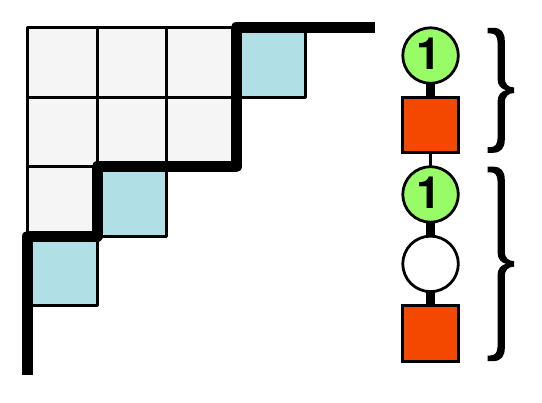}
  \caption{Bijection between Dyck paths and closed flows on linear
    trees. In this example, the flow has two connected components.}
  \label{fig:bijection}
\end{figure}

\begin{proposition}
  There exists a bijection $\rho$ between closed flows on $\linear_n$
  and Dyck paths of length $2n$ through which
  \begin{itemize}
  \item connected components correspond to indecomposable factors,
  \item outputs correspond to matched pairs of steps that do not form a
  peak.
  \end{itemize}
\end{proposition}
\begin{proof}
  The bijection is defined by induction on $n$. If $n=1$, there is
  only one closed flow on $\arb{0}$, which has no output, and only one
  Dyck path, which is $(0,1)(1,0)$.

  Assume now that $n$ is at least $2$, and the bijection $\rho$ is
  defined for smaller $n$.

  Any closed flow can be written as a list of connected components,
  starting from the component containing the root. Its image by $\rho$
  is defined as the concatenation of the images by $\rho$ of the
  connected components. 

  If there are at least $2$ connected components, this defines $\rho$
  by induction.

  If not, the closed flow is connected. Then the root is an
  output. One can remove $1$ to the rate of every edge and remove the
  root. This defines a closed flow on the linear tree with one vertex
  less. Its image by $\rho$ is taken to be $(0,1)D(1,0)$, where $D$ is
  the image by $\rho$ of the smaller flow, defined by induction.
 
  This decomposition is obviously mapped to the similar classical
  decomposition of Dyck paths, using sub-Dyck paths and down-moving of
  indecomposable paths. The inverse bijection is immediate.

  The statement on outputs follows easily by inspection of the
  bijection.
\end{proof}

The bijection is illustrated in figure \ref{fig:bijection}.

Let $\ca_{n,t}$ be the generating series $\xE_{\linear_n,t}$ and let
$\ca_{n}$ be the polynomial $\xE_{\linear_n}$.

The first few values of $\ca_{n,t}$ are
\begin{align*}
  \ca_{1,t}& =\frac{1}{1 - t},\quad
 \ca_{2,t}= \frac{1 + b - t b}{(1-t)^2}, \\
\ca_{3,t}&= \frac{1 + 3 b + b^{2}  - t (4 b + 2 b^{2})  + t^{2} (b  + b^{2})}{(1-t)^3}
\end{align*}

From the bijection above, it follows that $\ca_n$ counts Dyck paths
according to the number of peaks. These polynomials are classical in
combinatorics, and known as the Narayana polynomials, see for example
\cite{kostov_et_al}. We will call $\ca_{n,t}$ a $t$-Narayana fraction.

Let us introduce ordinary generating series
\begin{equation}
  E = \sum_{n \geq 1} \ca_n x^n \quad \text{and} \quad  E_t = \sum_{n \geq 1} \ca_{n,t} x^n,
\end{equation}
and let $E^c$ be the similar series for closed connected flows on linear
trees. 

The analogous series for small flows are just $x(1 + E)$ and $x(1
+ E_t)$, because a small flow on $\linear_{n+1}$ can be described by a
flow on $\linear_n$.

From the combinatorial decomposition used in the bijection with Dyck
paths, one deduces that
\begin{equation}
  \label{usual_eq}
  E=E^c/(1-E^c)\quad\text{and}\quad  E^c=x(1+b E).
\end{equation}

By decomposing a flow according to whether the root is an output or
not, one obtains the equation
\begin{equation*}
  E_t=x/(1-t)(1+E_t)+bx/t(E_t-E).
\end{equation*}
This is a special case of the global equation for flows
\eqref{master_eq_E}, that we will prove later.

It follows from all this that $E$ and $E^c$ are algebraic over
$\QQ(x)$ and that $E_t$ is algebraic over $\QQ(x,t)$.

\subsection{Conjectural formula for closed flows on forks}


Recall from \S \ref{section_pl} that $\fork_{i,n-i}$ is the fork
with $n$ vertices, with stem of size $i$.

\begin{conjecture}
  The number of closed flows of size $k$ on the fork $\fork_{i,n-i}$ is given by
  \begin{equation}
    \# \fl(\fork_{i,n-i},k,0)=\binom{i}{k}\binom{n}{k}-\binom{i+1}{k+1}\binom{n-1}{k-1}.
  \end{equation}
\end{conjecture}

For $i=n-1$ or $i=n$, corresponding to linear trees, this formula
gives the Narayana numbers, which is the correct result for the linear
trees (see section \ref{dyck}). One can easily check that this also
gives the correct answer for $i=1$, namely for corollas.

\subsection{Zeroes of flow polynomials}

After inspection of some examples, one is tempted to ask the following
question.

\begin{question}
  Let $T$ be a rooted tree. Are the zeroes of $\xF_T$ real and
  negative ? Are the zeroes of $\xE_T$ real and negative ?
\end{question}

It is known, for the Narayana polynomials, that all roots are real,
simple and negative, see for example \cite{kostov_et_al}. Therefore
the question has a positive answer for linear trees. One can also
check easily that this is true for corollas.

\section{Series of flows}

\subsection{Global equations for flows}

Let us introduce now two series 
\begin{equation}
  \xE_t=\sum_T \xE_{T,t} \frac{T}{\aut(T)} \quad\text{and}\quad \xD_t=\sum_T \xD_{T,t} \frac{T}{\aut(T)},
\end{equation}
in the group $\gp_{\prelie}$ associated with the Pre-Lie operad.

Let $\xE$ (resp. $\xD$) be the value at $t=0$ of $\xE_t$
(resp. $\xD_t$). 

\begin{theorem}
  \label{th_d_corolle_e}
  The following identity holds:
  \begin{equation}
  \label{eq_d_corolle_e}
    \xD_t= \COR \diam (\arb{0},\xE_t).
  \end{equation}
\end{theorem}
\begin{proof}
  This is essentially a restatement of Lemma
  \ref{relation_produit_D_E}, using the notation defined in
  \eqref{defi_crls} and the results of the appendix \ref{appA}.

  Namely, one applies Prop. \ref{diam_prop} of the appendix, with $A$
  the species of corollas, $B$ the species made only of the rooted
  tree on one vertex, $C$ the species of flows on rooted trees and $D$
  the species of small flows on rooted trees.
\end{proof}

\begin{theorem}
  \label{master_th_E}
  One has
  \begin{equation}
  \label{master_eq_E}
    \xE_t = \frac{1}{1-t}\xD_t + \frac{b}{t} \left( \xD_t - \xD \right).
  \end{equation}
\end{theorem}
\begin{proof}
  Consider a rooted tree $T=B_+(T_1,\dots,T_k)$ endowed with a flow.
  
  Either the root is an input vertex with multiplicity $\ell$ for some
  $\ell \geq 0$. This can be described using a small flow and the
  integer $\ell$. One obtains
  \begin{equation*}
    \frac{1}{1-t} \xD_t.
  \end{equation*}

  The other possibility is that the root is an output vertex. Removing
  the output, one gets a small flow with the condition that the exit
  rate is not zero. This gives the term
  \begin{equation*}
    \frac{b}{t} \left(\xD_t -  \xD\right).
  \end{equation*}
\end{proof}

\subsection{Global equations for connected flows}

Let $\xE^c_t$ be the global series of connected flows:
\begin{equation}
  \xE^c_t=\sum_T \xE^c_{T,t} \frac{T}{\aut(T)},
\end{equation}
and let $\xE^c$ be its value at $t=0$.

\begin{theorem}
  The series $\xE^c_t$ satisfies the following equation
  \begin{equation}
    \label{global_eq_connected}
    \xE^{c}_t = \frac{1}{1-t} \COR \diam (\arb{0},\xE^c_t-\xE^c) + \frac{b}{t} \left( \COR \diam (\arb{0}, \xE_t^c-\xE^c) - \arb{0} \right).
  \end{equation}
\end{theorem}
\begin{proof}
  This is similar to the proof of Theorems \ref{th_d_corolle_e}
  and \ref{master_th_E}. One has to distinguish according to the status
  of the root.

  If the root is an input (possibly empty), the restriction to every
  subtree is an arbitrary connected flow with non-zero outgoing rate
  at the root. We obtain the first term of the right-hand side.

  If the root is an output, there must be at least one subtree, and
  the restriction to every subtree is an arbitrary connected flow with
  non-zero outgoing rate at the root. This gives the second term of the
  right-hand-side.
\end{proof}

The series $\xE_t$ of flows can be recovered from the series $\xE^c_t$
of connected flows.

\begin{theorem}
  There holds
  \begin{equation}
    \label{rela_ect_e}
    \xE_t = \left( \sum_T \frac{T}{\aut(T)} \right) \diam( \xE^c_t, \xE^c).
  \end{equation}
\end{theorem}

\begin{proof}
  This follows from Prop. \ref{diam_prop} applied to the following four
  species: $A$ is the species of rooted trees, $B$ the species of
  connected flows, $C$ the species of closed connected flows and $D$
  the species of flows.

  The necessary bijection (\textbf{hypothesis} $H_\natural(A,B,C,D)$)
  is rather clear. Indeed, given any flow, one can define connected
  flows on its connected components, closed if not containing the
  root. One can also make a rooted tree $\tau$ with vertices the
  connected components. To be able to recover the flow, one has to
  know how to glue back components into the tree $\tau$. This is given
  by a constant of structure of the global composition of the Pre-Lie
  operad.
\end{proof}

In words, this theorem says that the series $\xE_t$ of flows is
obtained from the series of all trees, by insertion of $\xE^c_t$ in
the root and insertion of $\xE^c$ in all other vertices.

When $t=0$, this reduces to the factorisation of series
\begin{equation}
  \label{flow_is_tree_of_connected}
  \xE = \left( \sum_T \frac{T}{\aut(T)} \right) \circ \xE^c,
\end{equation}
in the group $\gp_{\prelie}$, which means that a closed flow is made
by gluing closed connected flows along a rooted tree.

\medskip

Because rooted trees that support closed connected flows have
root-valency at most $1$ by Lemma \ref{connexe_un}, one can write
\begin{equation}
  \label{from_EC_to_F}
  \xE^c= \arb{0}+ b\, \arb{0} \pl \xF,
\end{equation}
for some series $\xF$. We will use this series later on.

\subsection{Quotient series $\xE \circ \xD^{-1}$}

A \textbf{saturated flow} is a closed connected flow where every
non-leaf vertex is an output.

Let $\xE_T^s$ be the generating series for saturated flows on
$T$. Note that this is a monomial in the variable $b$, of degree the
number of non-leaf vertices of $T$.

\begin{lemma}
  \label{same_support}
  Let $T$ be a rooted tree that admits a closed connected flow. Then
  $T$ admits a saturated flow.
\end{lemma}
\begin{proof}
  Pick a closed connected flow on $T$. The proof is by induction on
  the number of non-leaf vertices which are not outputs. If the chosen
  flow is saturated, there is nothing to do. Otherwise, let $v$ be a
  non-leaf vertex which is not an output.

  If $v$ is not an input, one can put an output in $v$, choose a path
  from $v$ to some leaf $w$ of the subtree at $v$, and add $1$ to the
  rate on every edge of this path and $1$ input on $w$.

  If $v$ is an input, one can first move this input to a leaf, by
  choosing a path from $v$ to a leaf $w$ of the subtree at $v$, and
  adding $1$ to the rate on every edge of this path. Then one gets
  back to the previous case.
\end{proof}

Therefore rooted trees that admit closed connected flows are exactly
the same as rooted trees that admit saturated flows.

Let now $\xY$ be the quotient series $\xE \circ \xD^{-1}$ in the group
$\gp_{\prelie}$. One observes a surprising property.

\begin{conjecture}
  The coefficient $\xY_T$ of a rooted tree $T$ in $\xY$ is the monomial
  \begin{equation}
    (-1)^{L(T)-1} \xE_T^s,
  \end{equation}
  where $L(T)$ is the number of leaves of $T$.
\end{conjecture}

If this is true, then by Lemma \ref{same_support}, the support of
$\xY$ is the same as the support of $\xE^c$, and one can write
\begin{equation}
  \label{Z_here}
  \xY= \arb{0}+ b\, \arb{0} \pl \xZ,
\end{equation}
for some series $\xZ$. We will consider this series again later.




\section{Planar binary trees, dendriform operad and $\ncsf$}

\subsection{Notations for planar binary trees}

A \textbf{planar binary tree} on $n$ vertices is either the tree $1=|$ with no
inner vertex or a pair of two planar binary trees. Planar binary trees
will be drawn with their root at the bottom and leaves at the top,
aligned on a horizontal line. Examples are depicted in figure
\ref{fig:expl_bt}.

There is a natural involution on the set of planar binary trees, given
by left-right reversal, as shown in figure \ref{fig:expl_bt}.

The \textbf{canopy} of a planar binary tree is a sequence of letters
$\pp$ and $\mm$ of length $n-1$. There is a letter for each leaf but
the leftmost and rightmost one. The letter is $\mm$ is the leaf is the
left son of its parent vertex, and $\pp$ is the leaf is the right son
of its parent vertex.

For example, the canopy of the planar binary tree at the left of of figure
\ref{fig:expl_bt} is $\mm \pp \pp \mm \pp $.

We will also use the following variants: the full canopy is obtained
from the canopy by adding $\mm$ at the beginning and $\pp$ at the end,
the left-completed canopy by adding $\mm$ at the beginning, and the
right-completed canopy by adding $\pp$ at the end.

\begin{figure}
  \centering
  \includegraphics[height=2cm]{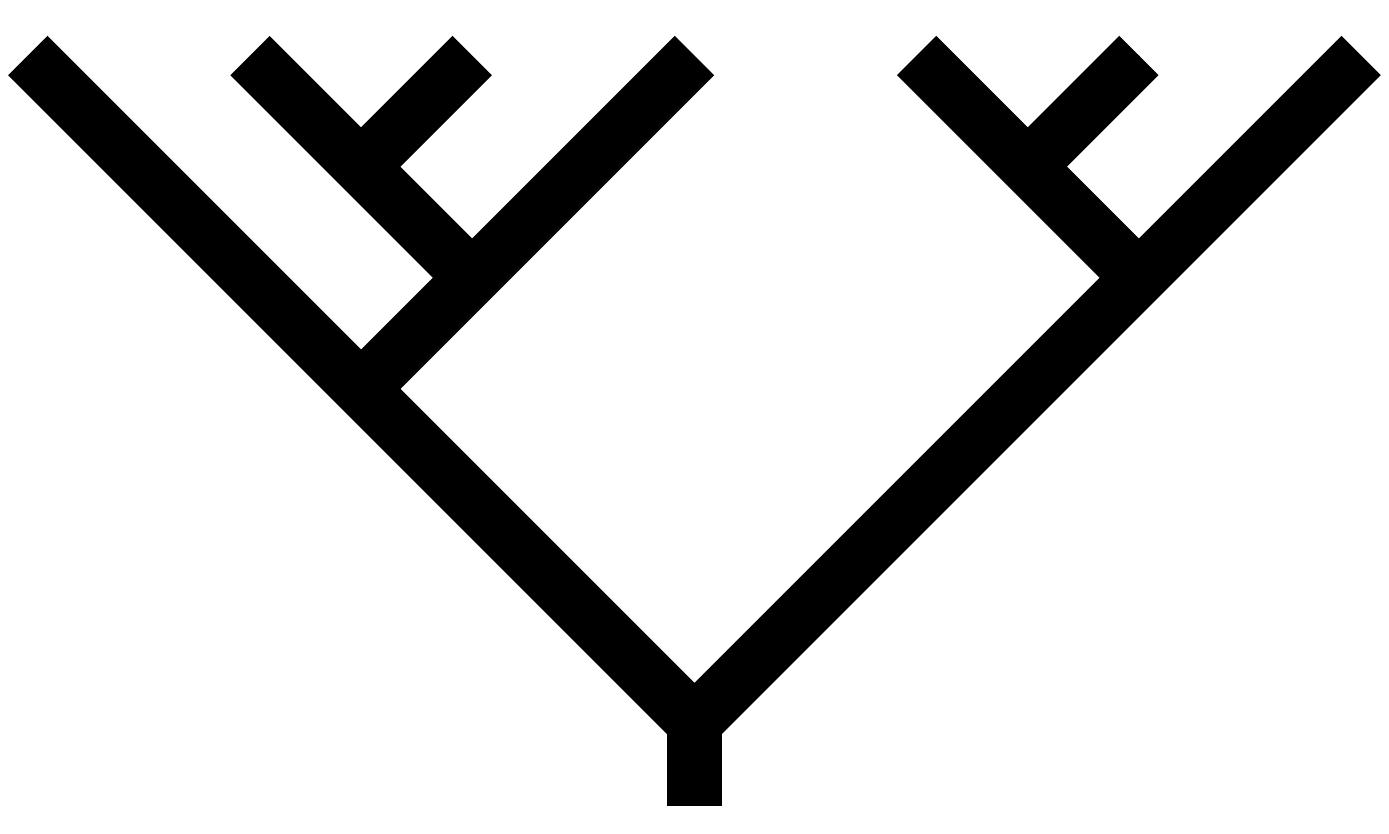}\hspace{1cm}\includegraphics[height=2cm]{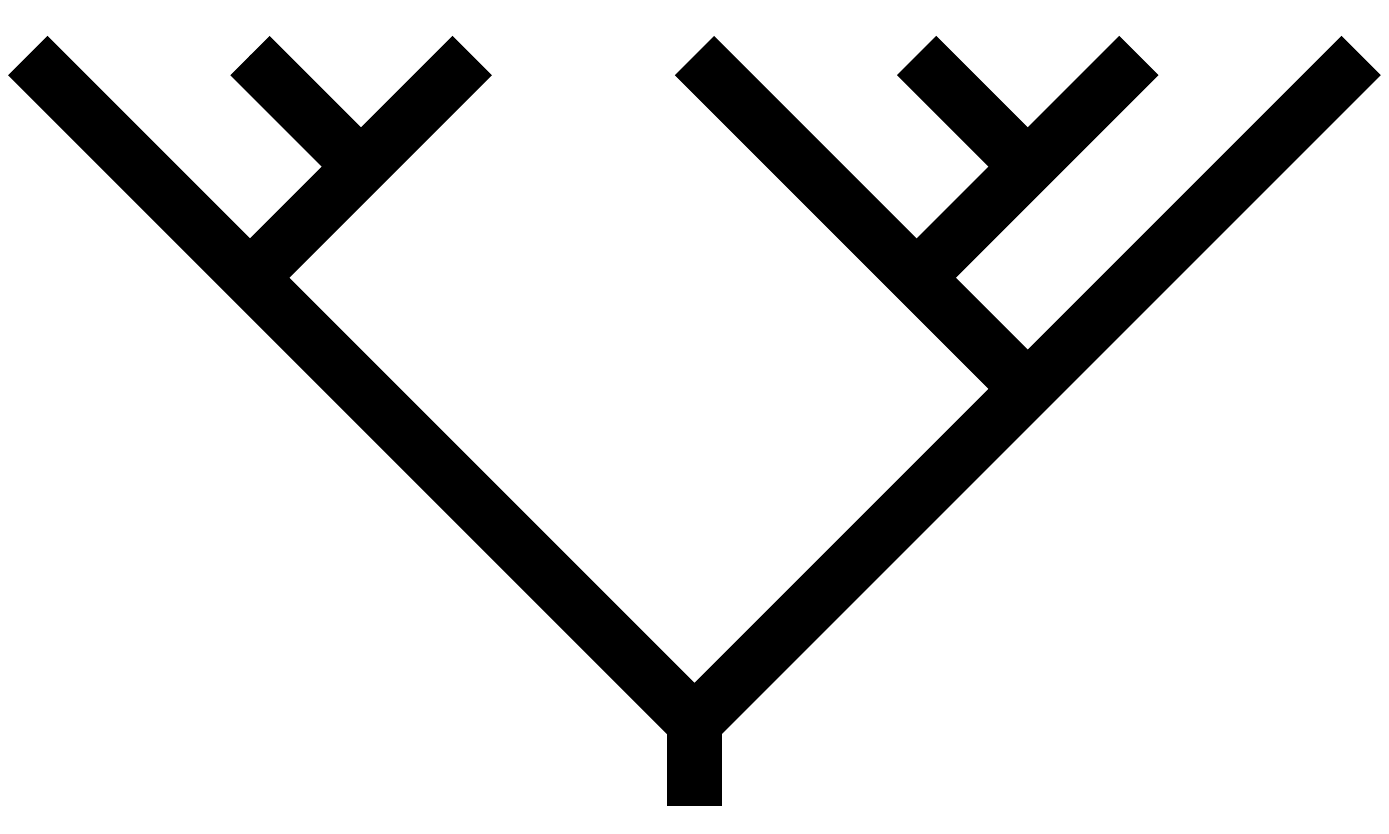}
  \caption{A planar binary tree with $6$ inner vertices, and its image by reversal.}
  \label{fig:expl_bt}
\end{figure}

\subsection{The Dendriform operad}

The Dendriform operad, introduced by Loday, is a non-symmetric operad
with a basis indexed by planar binary trees. The free dendriform
algebra is just the direct sum of all components of the Dendriform
operad. We refer the reader to \cite{loday_lnm} for more information
on the dendriform algebras.

On the free dendriform algebra $\dend$, there are two dendriform
products $\prec$ and $\succ$, that satisfy the $3$ dendriform
axioms. In particular, their sum defines an associative product
\begin{equation}
  x * y = x\succ y + x \prec y,
\end{equation}
which is the product used in the Hopf algebra structure of $\dend$.

We will use the following notation:
\begin{equation}
  \label{def_vee}
  x \vee_y z = x \succ y \prec z.
\end{equation}
By one of the dendriform axioms, no parentheses are needed in this
expression. When $y$ is the planar binary tree $\arb{1}$, the
operation $ x \vee z$ can be described as the gluing of $x$ and $z$ on
a common vertex.

From the dendriform axioms, one can deduce the following relations :
\begin{equation}
  \label{vee_star}
  (x \vee y) \prec z = x \vee (y*z) \quad \text{and} \quad  x \succ (y \vee z) = (x*y) \vee z.
\end{equation}

Oen can extend (in a unique way) the notation $x \vee_y z$ to the
cases where $x$ or $z$ are the unit tree $1$, with the same
properties.

Let $\phi$ be the operad morphism from the $\prelie$ operad to the
$\dend$ operad defined by its value on the labelled generator:
\begin{equation}
  \label{defi_phi}
  \phi(x \pl y) =y \succ x - x \prec y. 
\end{equation}
Therefore, the map $\phi$ sends $\arb{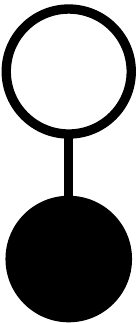}$ to $\arb{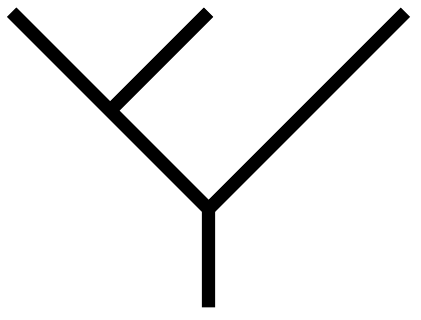}-\arb{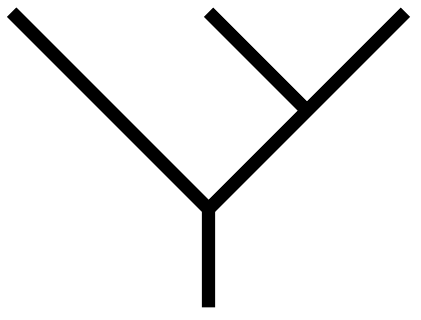}$. One
can show that the morphism $\phi$ is injective by using that it
factorises through the Brace operad.

From now on, the expression ``dendriform image'' will mean the image
by $\phi$.

We will work in the group $\gp_{\dend}$ associated with the dendriform
operad. This is an open subset in the free dendriform algebra on one
generator $\dend$.

\begin{lemma}
  \label{racine_map}
  Let $T$ be a labeled rooted tree, with $i$ the label of the root.
  The dendriform image of $T$ is a linear combination of labelled planar binary
  trees whose root is labeled by $i$.
\end{lemma}
\begin{proof}
  One has to show that $\phi$ is a morphism of rooted-operads, in the
  language of the appendix \ref{appA}. This is clear on the generators
  by \eqref{defi_phi}, hence one can apply Lemma \ref{check_on_gen}.




\end{proof}

\begin{lemma}
  \label{diamant_vee}
  Let $x,y,z,t$ in $\dend$. Then
  \begin{equation}
    (x \vee y) \diamond (z,t) = (x \circ t) \vee_z (y \circ t).
  \end{equation}
\end{lemma}
\begin{proof}
  This is an easy consequence of the definition \eqref{def_vee} of
  $\vee$ and of the general definition of the operations $\diamond$
  and $\circ$ in appendix \ref{appA}.
\end{proof}

\begin{lemma}
  \label{double_inversion}
  Let $x,y,z,t, u$ in $\dend$. Let $   v=(y \vee_z t).$  Then
  \begin{equation*}
    x \vee_{v} u = (x * y) \vee_z (t * u). 
  \end{equation*}
\end{lemma}
\begin{proof}
  This is a simple computation in the dendriform operad, starting from
  the definition \eqref{def_vee}.
\end{proof}

The suspension $\Sigma$ is defined by
\begin{equation}
  \Sigma \big{(} \sum_{n \geq 1}a_n \big{)} = \sum_{n \geq 1} (-1)^{n-1} a_n,
\end{equation}
where $a_n$ is homogeneous of degree $n$.

We will also use the \textbf{bar involution}, which is the composition of
suspension and reversal, that are two commuting involutions.

\subsection{The subalgebra $\ncsf$ of $\dend$}

Let $\ncsf$ be the algebra of non-commutative symmetric
functions. This is the free associative algebra generated by one
generator in every positive degree. We will use the basis of ribbon
Schur functions, indexed by compositions of $n$ in degree $n$. For
more information, the reader may consult
\cite{ncsf, thibon_lectures, ncsf6}.

Compositions of $n$ will be identified with strings of $n-1$ symbols
$\pp$ and $\mm$, by the convention that a $\mm$ symbol means ``cut
here'' and a $\pp$ symbols means ``do not cut here''. For example,
\begin{equation}
  1|4|1|2 \longleftrightarrow \mm \pp \pp \pp \mm \mm \pp.
\end{equation}

The product in the basis of ribbon Schur functions is given by the rule
\begin{equation}
  \epsilon * \delta = \epsilon\pp\delta + \epsilon\mm\delta.
\end{equation}

The inclusion from $\ncsf$ to $\dend$ is defined on the basis of Schur
function by sending a sequence of elements of $\{\pp,\mm\}$ to the sum
of all planar binary having this sequence as canopy. This is a
morphism of algebras.

One will need the following lemma.
\begin{lemma}
  \label{coeff_id}
  Let $\theta$ be a Lie idempotent in the descent algebra of $\sym_n$,
  seen as an element of $\ncsf$. Then the coefficient of the ribbon
  Schur function with index $\pp^{n-1}$ in $\theta$ is $1/n$.
\end{lemma}
\begin{proof}
  By \cite[Prop. 2.4]{schocker}, the coefficient of $\Id$ in the
  expansion of any Lie idempotent in the usual basis of the symmetric
  group ring $\QQ[\sym_n]$ is $1/n$. By the inclusion of $\ncsf$ in
  $\fqsym$, a ribbon Schur function is mapped to the sum of all
  permutations with a fixed descent set, depending on its index. For
  the ribbon Schur function with index $\pp^{n-1}$, the image is just
  the permutation $\Id$. For a Lie idempotent in the descent algebra,
  the coefficient of the ribbon Schur function with index $\pp^{n-1}$
  is therefore $1/n$.
\end{proof}

\subsection{Known series in the dendriform group}

Let us recall some elements of $\gp_{\dend}$ and their properties.

Let $\sR$ be the positive sum of all right combs,
\begin{equation*}
  \sR = \dun + \arb{21}+ \arb{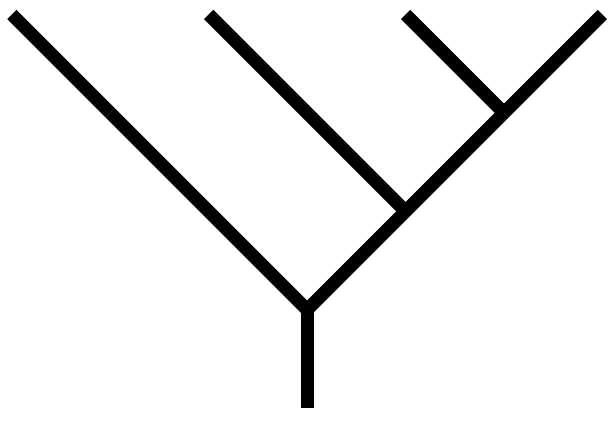} + \dots 
\end{equation*}
This is the unique solution of the equation
\begin{equation}
  \label{carac_R}
  \sR = \dun + \dun \prec \sR.
\end{equation}

Let $\sL$ be the alternative sum of all left combs,
\begin{equation*}
  \sL = - \dun +  \arb{12} - \arb{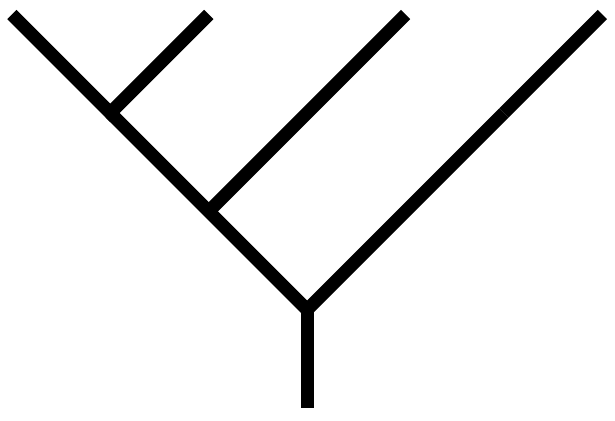} + \dots 
\end{equation*}
This is the unique solution of the equation
\begin{equation}
  \label{carac_L}
  \sL=- \dun - \sL \succ \dun.
\end{equation}

The bar involution maps $\sR$ to $-\sL$.

\begin{lemma}
  \label{L_inverse_R}
  The following inversion relation holds:
  \begin{equation}
    (1+\sL) * (1+ \sR)=1.
  \end{equation}
\end{lemma}
\begin{proof}
  Both $1+\sL$ and $1+\sR$ belong to the subalgebra $\ncsf$. Indeed
  $\sL$ and $\sR$ are the same as 
  \begin{equation*}
    \sum_{k \geq 0} (-1)^{k+1} \pp^k \quad \text{and}\quad \sum_{k \geq 0} \mm^k.
  \end{equation*}
  With the product rule $A * B = A \pp B+A \mm B$ of $\ncsf$, this
  identity is easily proved there. Another proof can be found in
  \cite[Prop. 5.1]{qidempotent}.
\end{proof}

\begin{proposition}
  \label{image_des_corolles}
  The dendriform image of 
  \begin{equation}
    \COR = \sum_{n\geq 0} \frac{\corol_n}{n!} \quad \text{is} \quad (1+\sR) \vee (1+\sL).
  \end{equation}
\end{proposition}
\begin{proof}
  This was proved in \cite{ronco}.
\end{proof}

We will now recall and extend some results of \cite{qidempotent}. 
Beware that this article uses slightly different notations.

Recall from section \ref{section_pl} that $\linear_\ell$ is the linear
rooted tree with $\ell$ vertices.

\begin{lemma}
  \label{lemme_auxi_1}
  The dendriform image of $\sum_T \frac{T}{\aut(T)}$ is given by
  \begin{equation}
    (1-\Sigma \sL) * \phi\left(\sum_{\ell \geq 1} \linear_\ell \right) * (1-\Sigma \sR).
  \end{equation}
\end{lemma}
\begin{proof}
  This follows from \cite[Prop. 5.6]{qidempotent} (at $q=\infty$) and \cite[Prop. 6.4]{qidempotent}. One also uses Lemma \ref{L_inverse_R}.
\end{proof}

\begin{lemma}
  \label{lemme_auxi_2}
  One has
  \begin{equation}
    (1-\Sigma \sL) * \phi\left(\sum_{\ell \geq 1} \linear_\ell \right)=
    \sum_{n \geq 1} n L_n,
  \end{equation}
  where $L_n$ is the left comb with $n$ vertices.
\end{lemma}
\begin{proof}
  This is essentially \cite[Prop. 5.3]{qidempotent}.
\end{proof}

\begin{lemma}
  \label{lemme_auxi_3}
  There holds
  \begin{equation}
    \left(\sum_{n \geq 1} n L_n \right)* (1-\Sigma \sR) = (1-\Sigma\sL)\vee (1-\Sigma\sR).
  \end{equation}
\end{lemma}
\begin{proof}
  This is a simple computation in the dendriform algebra, or even in
  the sub-algebra $\ncsf$, with easy cancellations.
\end{proof}

\begin{proposition}
  \label{image_de_tout}
  The dendriform image of 
  \begin{equation}
    \sumt_1 = \sum_T \frac{T}{\aut(T)} \quad \text{is} \quad (1- \Sigma \sL) \vee (1- \Sigma \sR).
  \end{equation}
\end{proposition}
\begin{proof}
  This follows from the Lemmas \ref{lemme_auxi_1},
  \ref{lemme_auxi_2} and \ref{lemme_auxi_3}.
\end{proof}

\begin{lemma}
  \label{inverse_cool}
  In $\gp_{\dend}$, the inverse of $(1- \Sigma \sL) \vee (1- \Sigma
  \sR)$ is $(1+ \sL) \vee (1+ \sR)$.
\end{lemma}
\begin{proof}
  By lemma \ref{inversion_somme_tous}, the inverse of $\sumt_1$ is
  $\sumt_{-1}$, which is the suspension of $\sumt_1$.

  The result then follows from proposition \ref{image_de_tout}, by
  functoriality of the group construction.
\end{proof}

\begin{lemma}
  \label{nice_compo}
  One has
  \begin{equation}
    \Sigma \sR \circ ((1+\sL) \vee (1+\sR)) = -\sL.
  \end{equation}
\end{lemma}
\begin{proof}
  By equation \eqref{carac_L}, it is enough to prove that
  \begin{equation*}
    \Sigma\sR \circ \left((1+\sL)\vee (1+\sR)\right) = \dun - (\Sigma\sR \circ \left((1+\sL)\vee (1+\sR)\right)) \succ \dun.
  \end{equation*}
  By composition with the inverse of $(1+\sL)\vee (1+\sR)$ given by
  lemma \ref{inverse_cool}, this is equivalent to
  \begin{equation*}
    \Sigma\sR = (1-\Sigma \sL)\vee (1-\Sigma \sR) - \Sigma\sR \succ \left((1-\Sigma \sL)\vee (1-\Sigma \sR) \right).
  \end{equation*}
  By suspension, this is the same as
  \begin{equation*}
    \sR = (1+ \sL)\vee (1+ \sR) + \sR \succ \left((1+ \sL)\vee (1+ \sR) \right).
  \end{equation*}
  By Lemma \ref{L_inverse_R} and \eqref{vee_star}, this is equivalent to
  \begin{equation*}
    \sR = 1 \vee (1 + \sR),
  \end{equation*}
  which is just the equation \eqref{carac_R}.
\end{proof}

\section{Series in $\ncsf$} 

Let $\dP_t$ and $\dN_t$ be series in variables $\pp,\mm$ defined by
\begin{equation}
  \dP_t=\sum_{k \geq 1} \ca_{k,t} \pp^k\quad \text{and}\quad  \dN_t=\sum_{k \geq 1} (-1)^k \ca_{k,t} \mm^k,
\end{equation}
where $ca_{k,t}$ are the $t$-Narayana fractions defined in \S \ref{dyck}, and let $\dP$ (resp. $\dN$) be $\dP_{t=0}$ (resp. $\dN_{t=0}$). 

These series can be considered as ordinary generating series for flows
on linear trees, see section \ref{dyck}.

\begin{lemma}
  \label{equation_pt_nt}
  One has
  \begin{equation}
    \label{defi_pt}
    \dP_t = \frac{1}{1-t} (\emptyset+\dP_t) \pp +\frac{b}{t} (\dP_t - \dP) \pp
  \end{equation}
  and
  \begin{equation}
    \dN_t = \frac{-1}{1-t} \mm (\emptyset+\dN_t) -\frac{b}{t} \mm (\dN_t - \dN).
  \end{equation}
\end{lemma}
\begin{proof}
  This follows from the fact that the coefficients $\ca_{k,t}$ count
  flows on linear rooted trees. One has to decompose according to
  whether the root is an output or not, as already done in the proof
  of Theorem \ref{main}.
\end{proof}

One will also need connected variants of $\dP$ and $\dN$, defined by
\begin{equation}
  \label{defi_pc_nc}
  \dP^c= (\emptyset+b\dP)\pp \quad \text{and}\quad \dN^c= -\mm(\emptyset+b\dN).
\end{equation}

By \eqref{usual_eq}, these series are generating series for connected
closed flows on linear trees. Let us now consider the similar series
$\dP^c_t$ for arbitrary connected flows on linear trees.

Every flow on a linear tree can be decomposed as a list of connected
components, all but one are closed. One therefore has
\begin{equation}
  \label{decoupe_flot_lineaire}
  \dP_t=\dP^c_{t}+\dP_t \dP^c.
\end{equation}

A connected flow on a linear tree is either closed, or one can remove one layer of rate on every edge, and obtain any linear flow. This implies that
\begin{equation}
  \label{flot_lineaire_connexe}
  \dP^c_t = \dP^c + t \dP_t.
\end{equation}

\begin{proposition}
  The series $\dP^c$ and $\dN^c$ satisfy
  \begin{equation}
    \label{prop_pc_nc}
    (1-t) \dP_t=\dP^c+ \dP^c \dP_t \quad \text{and} \quad  (1-t) \dN_t=\dN^c+ \dN_t \dN^c.
  \end{equation}
\end{proposition}
\begin{proof}
  It is enough to consider the case of $\dP$, by symmetry under the
  exchange of $\pp$ and $-\mm$. The equation follows directly from
  \eqref{decoupe_flot_lineaire} and \eqref{flot_lineaire_connexe}.
\end{proof}

Let us now define three series involving both variables $\pp$ and
$\mm$.

The series $\dT$ is defined by
\begin{equation}
  \label{defi_dt}
  \dT=\sum_{k \geq 0} b^k (\dP \dN)^k,
\end{equation}
and does not depend on the variable $t$.

The series $\dU_t$ is then defined by
\begin{equation}
  \label{def_u}
  \dU_t=(\emptyset+b \dN)\dT\dP_t.
\end{equation}
The first few terms of $\dU_t$ are 
\begin{equation*}
  \ca_{1,t} \pp + \ca_{2,t} \pp\pp -b\ca_{1,t}\mm\pp + \dots
\end{equation*}

The series $\dV_t$ is similarly defined by
\begin{equation}
  \label{def_v}
  \dV_t=\dN_t \dT (\emptyset+b \dP) .
\end{equation}
Its first few terms are 
\begin{equation*}
  - \ca_{1,t} \mm + \ca_{2,t} \mm\mm  - b\ca_{1,t}\mm\pp + \dots
\end{equation*}

Let $\dU$ (resp. $\dV$) be $\dU_{t=0}$ (resp. $\dV_{t=0}$).

\begin{lemma}
  One has
  \begin{equation}
    \label{droite_u_plus}
    \dU_t = \frac{1}{1-t}  ((\emptyset+b\dN)\dT+\dU_t) \pp +\frac{b}{t}  (\dU_t - \dU) \pp
  \end{equation}
  and
  \begin{equation}
    \dV_t = \frac{-1}{1-t} \mm (\dT(\emptyset+b\dP)+\dV_t) -\frac{b}{t} \mm (\dV_t - \dV).
  \end{equation}
\end{lemma}
\begin{proof}
  This follows from Lemma \ref{equation_pt_nt} and the definition of
  $\dU_t$ and $\dV_t$.
\end{proof}

One will need an involution (called the \textbf{bar involution}) on
the space of non-commutative formal power series in two variables
$\mm,\pp$. It is the unique anti-morphism of algebra defined on
generators by $\overline{\pp}=-\mm$ and $\overline{\mm}=-\pp$.

Under the bar involution, $\dN_t$ and $\dP_t$ are exchanged, $\dT$ is
fixed and $\dU_t$ and $\dV_t$ are exchanged.

\section{Series in the dendriform group}

Let $\sU_t$ be the
unique dendriform series whose right-completed canopy is given by $\dU_t$ :
\begin{equation}
  \sU_t= \ca_{1,t} \dun + \ca_{2,t} \arb{12}  -b\ca_{1,t}\arb{21} + \dots
\end{equation}
and let $\sV_t$ be the unique dendriform series whose left-completed canopy is given
by $\dV_t$ :
\begin{equation}
  \sV_t= - \ca_{1,t} \dun + \ca_{2,t} \arb{21}  - b\ca_{1,t}\arb{12} + \dots. 
\end{equation}

It follows from this definition that, under the bar involution on
dendriform series, one has $\overline{\sU}_t=-\sV_t$.

\begin{lemma}
  \label{lemmeVU}
  Let $u,v$ be two indeterminates. One has
  \begin{equation}
    (1+\sV_v) * (1+ \sU_u) = 1 + (v-u) \dN_v T \dP_u,
  \end{equation}
  where $\dN_v T \dP_u$ has to be interpreted as the sum over planar
  binary trees with the given full canopy.
\end{lemma}
\begin{proof}
  This is in fact a computation inside series in $\pp$ and $\mm$, by
  the correspondence between a monomial in $\pp$ and $\mm$ and the sum
  of all planar binary trees having this monomial as their full
  canopy.
  
  Let us compute $(1+\sV_v) * (1+ \sU_u) -1$. One finds
  \begin{multline*}
    \dN_v\dT(\emptyset+b\dP)\pp + \mm (\emptyset + b\dN) \dT \dP_u + \dN_v\dT(\emptyset+b\dP)\pp(\emptyset + b\dN) \dT \dP_u \\+\dN_v\dT(\emptyset+b\dP) \mm (\emptyset + b\dN) \dT \dP_u .
  \end{multline*}
  Using the definition \eqref{defi_pc_nc} of $\dN^c$ and $\dP^c$, one gets
  \begin{equation*}
    \dN_v\dT\dP^c -\dN^c \dT \dP_u + \dN_v\dT\dP ^c (\emptyset + b\dN) \dT \dP_u -\dN_v\dT(\emptyset+b\dP)\dN^c \dT \dP_u .
  \end{equation*}
  Expanding the products, one obtains
  \begin{multline*}
    \dN_v\dT\dP^c -\dN^c \dT \dP_u + \dN_v\dT\dP ^c \dT \dP_u+ b \dN_v\dT\dP ^c \dN \dT \dP_u \\ -\dN_v\dT\dN^c \dT \dP_u-b \dN_v\dT \dP\dN^c \dT \dP_u  .
  \end{multline*}
  One can then use the fact that $\dT=\emptyset+b \dT \dP \dN=\emptyset+b \dP \dN \dT$ to split the third and fifth terms, getting
  \begin{multline*}
     \dN_v\dT\dP^c -\dN^c \dT \dP_u + \dN_v\dT \dP ^c \dP_u+ b \dN_v\dT \dP ^c \dP \dN \dT \dP_u+ b \dN_v\dT\dP ^c \dN\dT \dP_u \\-\dN_v\dN^c \dT \dP_u -b \dN_v\dT \dP \dN \dN^c \dT \dP_u-b \dN_v\dT \dP\dN^c \dT \dP_u .   
  \end{multline*}
  Now gathering terms by pairs and using four times the equation
  \eqref{prop_pc_nc}, one gets, after some
  cancellations,
  \begin{equation*}
    (1-u) \dN_v \dT \dP_u - (1-v) \dN_v \dT \dP_u,
  \end{equation*}
  which is the expected result.
\end{proof}

\subsection{Flows in the dendriform group}

Let us now consider two series $\sD_t$ and $\sE_t$. Our aim will be to
show that they are the respective dendriform images of the series
$\xD_t$ and $\xE_t$.

The series $\sD_t$ is defined by
\begin{equation}
  \label{defi_dt}
  \sD_t= (1+\sU_t)\vee(1+\sV_t),
\end{equation}
and its first few terms are
\begin{equation*}
  \dun + \ca_{1,t} \arb{12} - \ca_{1,t} \arb{21}  +\dots
\end{equation*}
The series $\sD$ is the value of $\sD_t$ at $t=0$.

The series $\sE_t$ is then defined by
\begin{equation}
  \label{defi_et}
  \sE_t =\frac{1}{1-t} \sD_t + \frac{b}{t} (\sD_t-\sD).
\end{equation}
Its first few terms are
\begin{equation*}
\ca_{1,t} \dun + \ca_{2,t} \arb{12} - \ca_{2,t} \arb{21} + \dots
\end{equation*}
The series $\sE$ is the value of $\sE_t$ at $t=0$.

From these definitions, it results that both $\sE_t$ and $\sD_t$ are
fixed under the bar involution of $\dend$.

\begin{proposition}
  \label{uv_from_et}
  One has the following relations
  \begin{equation}
    \sU_t =  \sR \circ \sE_t \quad \text{and}\quad    \sV_t =  \sL \circ \sE_t.
  \end{equation}
\end{proposition}
\begin{proof}
  By symmetry under the bar involution, it is enough to prove the
  first equation. By the characteristic property \eqref{carac_R} of
  right combs, one just has to show that
  \begin{equation*}
    \sU_t = \sE_t + \sE_t \prec \sU_t.
  \end{equation*}
  Let us compute the right hand side using \eqref{defi_et}. One finds
  \begin{equation*}
    \frac{1}{1-t} \sD_t + \frac{b}{t} (\sD_t-\sD) + \frac{1}{1-t} \sD_t \prec \sU_t + \frac{b}{t} (\sD_t\prec \sU_t-\sD\prec \sU_t).
  \end{equation*}
  Using then \eqref{defi_dt}, one gets
  \begin{multline*}
    \frac{1}{1-t} (1+\sU_t)\vee(1+\sV_t) + \frac{b}{t}
    ((1+\sU_t)\vee(1+\sV_t) - (1+\sU)\vee(1+\sV) )+ \\ \frac{1}{1-t}
    ((1+\sU_t)\vee(1+\sV_t)) \prec \sU_t + \frac{b}{t} ((
    (1+\sU_t)\vee(1+\sV_t))\prec \sU_t-( (1+\sU)\vee(1+\sV))\prec
    \sU_t),
  \end{multline*}
  which can be rewritten by \eqref{vee_star} as
  \begin{multline*}
    \frac{1}{1-t} (1+\sU_t)\vee((1+\sV_t)*(1+\sU_t)) + \frac{b}{t}
    ((1+\sU_t)\vee((1+\sV_t)*(1+\sU_t)) \\ - (1+\sU)\vee ((1+\sV)*(1+\sU_t) ).
  \end{multline*}
  Using Lemma \ref{lemmeVU}, one can replace $(1+\sV_t)*(1+\sU_t)$
  by $1$. One obtains
  \begin{equation*}
    \frac{1}{1-t} (1+\sU_t)\vee 1 +\frac{b}{t}(1+\sU_t)\vee 1 -\frac{b}{t} (1+\sU)\vee ((1+\sV) * (1+\sU_t)). 
  \end{equation*}
  Using Lemma \ref{lemmeVU} again, one finds
  \begin{equation*}
    \frac{1}{1-t} (1+\sU_t)\vee 1+\frac{b}{t}(1+\sU_t)\vee 1 - \frac{b}{t} (1+\sU)\vee (1-t(\dN T \dP_t)). 
  \end{equation*}
  Expanding that, one gets
  \begin{equation}
    \label{pas_final}
    \frac{1}{1-t} 1\vee 1 +\frac{1}{1-t} \sU_t\vee 1+\frac{b}{t} \sU_t\vee 1 - \frac{b}{t} \sU \vee 1 + b  1 \vee (\dN T \dP_t)+ b \sU \vee (\dN T \dP_t). 
  \end{equation}
  
  We therefore have to show that this expression is simply $\sU_t$.

  To prove that, let us decompose $\sU_t$ according to the position of
  the root in the trees. There are four ways to place the root in the
  full canopy:
  \begin{itemize}
  \item the tree is $\dun$, the root can be put between $\mm$ and $\pp$,
  \item the full canopy ends by $\pp\pp$, the root can be put between them,
  \item the full canopy starts by $\mm\mm$, the root can be put
    between them,
  \item the root can be put after any $\pp$ followed by $\mm$ in the
    full canopy.
  \end{itemize}

  Using equation \eqref{droite_u_plus}, let us describe the first two cases. One gets
  \begin{itemize}
  \item $\frac{1}{1-t} 1 \vee 1$ for the tree $\dun$,
  \item $\frac{1}{1-t} \sU_t\vee 1+\frac{b}{t} \sU_t\vee 1 - \frac{b}{t} \sU \vee 1 $ for the root between $\pp$ and $\pp$.
  \end{itemize}
  
  Using equation \eqref{def_u}, let us describe the last two cases. One gets
  \begin{itemize}
  \item $b  1 \vee (\dN T \dP_t)$ for the root between $\mm$ and $\mm$,
  \item  $b \sU \vee (\dN T \dP_t)$ for the root between $\pp$ and $\mm$.
  \end{itemize}
  Note that the last case is slightly more subtle, as the cut takes
  places inside the $\dT$ factor and one has to use the expression
  \eqref{defi_dt} for the series $\dT$.
  
  It follows that \eqref{pas_final} is exactly the expansion of
  $\sU_t$ according to the possible positions of the root in the
  canopy.
\end{proof}

\begin{corollary}
  One has 
  \begin{equation}
    \sD_t= ((1+\sR) \circ \sE_t) \vee ((1+\sL)\circ \sE_t).
  \end{equation}
\end{corollary}
\begin{proof}
  This follows from \eqref{defi_dt} and Proposition \ref{uv_from_et}.
\end{proof}

This is readily reformulated by lemma \ref{diamant_vee} using the
$\diamond$ operation as
\begin{equation}
  \label{d_corolle_e_dend}
  \sD_t= ((1+\sR) \vee (1+\sL)) \diamond (\arb{1},\sE_t).
\end{equation}

\begin{theorem}
  The dendriform images of $\xE_t$ and $\xD_t$ are $\sE_t$ and $\sD_t$.
\end{theorem}
\begin{proof}
  The series $\xE_t$ and $\xD_t$ are characterized by the equations
  \eqref{master_eq_E} and \eqref{eq_d_corolle_e}. By Proposition
  \ref{image_des_corolles} and results of appendix \ref{appA}, the
  dendriform image of \eqref{eq_d_corolle_e} is exactly
  \eqref{d_corolle_e_dend}. The dendriform image of
  \eqref{master_eq_E} is exactly \eqref{defi_et}. Therefore $\sE_t$
  and $\sD_t$ satisfy equations that characterize the dendriform
  images of $\xE_t$ and $\xD_t$, and the statement follows.
\end{proof}

\subsection{Explicit product formulas for coefficients}

The equations \eqref{def_u} and \eqref{def_v} provide an explicit
description of the coefficients of the series $\sU_t$ and $\sV_t$.

More precisely, the coefficient of a planar binary tree $\tau$ in the
series $\sU_t$ can be found as follows. One considers the
right-completed canopy of $\tau$ (including the rightmost leaf but not
the leftmost leaf). It admits a unique coarsest decomposition into
blocks of the shape $\pp^k$ and $\mm^\ell$ for $k,\ell \geq 1$. Every
block of length $\ell$ in this decomposition contributes a Narayana
factor $\ca_\ell$, but the rightmost block contributes instead a
$t$-Narayana factor $\ca_{\ell,t}$. There is an additional factor of
$b$ to the power the number of $\mm$ blocks and $(-1)$ to the power
the number of $\mm$.

For example, the coefficient of the leftmost planar binary tree of
figure \ref{fig:expl_bt}, whose right-completed canopy is
$\mm\pp\pp\mm\pp\pp$, is
\begin{equation*}
  b^2 \ca_1 \ca_2 \ca_1 \ca_{2,t}.
\end{equation*}

There is a similar description for $\sV_t$. One considers the
left-completed canopy of $\tau$ (including the leftmost leaf but not
the rightmost leaf) and decompose it into maximal blocks of $\pp$ and
$\mm$. Every such block of length $\ell$ contributes a Narayana factor
$\ca_\ell$, but the leftmost block contributes instead a $t$-Narayana
factor $\ca_{\ell,t}$. There is an additional factor of $b$ to the
power the number of $\pp$ blocks and $(-1)$ to the power the number of
$\mm$.

\medskip

One can also interpret the definition \eqref{defi_dt} as giving the
explicit coefficients of the series $\sD_t$.

More precisely, the coefficient of a planar binary tree $\tau$ in the
series $\sD_t$ can be found as follows. Consider the canopy of $\tau$,
and cut it into two parts according to the position of the root of
tree. Decompose both parts into maximal blocks of $\pp$ and $\mm$. Every such
block of length $\ell$ contributes a Narayana factor $\ca_\ell$, but
the two blocks that are closest to the root contributes instead a
$t$-Narayana factor $\ca_{\ell,t}$. There is an additional factor of
$b$ to the power the number of $\mm$ blocks in the left part plus the
number of $\pp$ blocks in the right part, and $(-1)$ to the power the
number of $\mm$.

For example, the coefficient of the leftmost planar binary tree of
figure \ref{fig:expl_bt}, whose canopy is cut into
$\mm\pp\pp$ and $\mm\pp$, is
\begin{equation*}
  b^2 \ca_1 \ca_{2,t} \ca_{1,t} \ca_{1}.
\end{equation*}

Letting $t=0$ in this description, one observes that the coefficient
of a tree in $\sD$ depends only on its canopy. This is obvious for the
factors associated with blocks and for the sign. As for the power of
$b$, it can be described as the number of $\mm$ blocks in the canopy,
excluding the last (rightmost) block.

It follows that $\sD$ is in the descent algebra. Moreover, this
description of $\sD_n$ is exactly the value at $a=1$ of the
description given in \cite[Th. 10.1]{menoth} and we therefore recover
this theorem. Let us give its statement here.

\begin{corollary}
  The homogeneous components $\sD_n$ are Lie idempotents and satisfy
  \begin{equation}
    \sD_n \cdot \sD_n = n \ca_{n-1} \sD_n,
  \end{equation}
  in the symmetric group ring of $\sym_n$.
\end{corollary}

To determine the precise constant of proportionality, one uses Lemma
\ref{coeff_id} and the fact that the coefficient of the ribbon Schur
function $\pp^{n-1}$ is $\ca_{n-1}$.

\medskip

One can now use \eqref{defi_et} to give an explicit description of the
coefficients of the series $\sE_t$.

As $\sD_t$ and $\sD$ have all but two of their factors in common, all
these factors are also in $\sE_t$. The remaining factor is
\begin{equation}
  \frac{1}{1-t}\ca_{k,t} \ca_{\ell,t}+\frac{b}{t}\left(\ca_{k,t} 
    \ca_{\ell,t}-\ca_k \ca_{\ell} \right)
\end{equation}
which is the fraction counting flows on the rooted trees
$B_+(\linear_k,\linear_{\ell})$.

Therefore, the coefficient of a planar binary tree $\tau$ in the
series $\sE_t$ can be found as follows. Consider the canopy of $\tau$,
and cut it into two parts according to the position of the root of
tree. Decompose both parts into maximal blocks of $\pp$ and $\mm$,
excluding the two central blocks. Every such block of length $\ell$
contributes a Narayana factor $\ca_\ell$. The two central blocks
together are of the shape $\pp^k \mm^\ell$. We associate with this the
coefficient of the rooted tree $B_+(\linear_k,\linear_{\ell})$ in the
series $\xE_t$.

There is an additional factor of $b$ to the power the number of $\mm$
blocks in the left part plus the number of $\pp$ blocks in the right
part, and $(-1)$ to the power the number of $\mm$.

\subsection{Connected flows in the dendriform group}

Let us now introduce the dendriform image $\sE^c_t$ of the series
$\xE^c_t$ of connected flows. As $\xE^c_t$ is related to $\xE_t$ by
equation \eqref{rela_ect_e}, one gets, by using Proposition
\ref{image_de_tout} and results of appendix \ref{appA}, that $\sE^c_t$
is defined by
\begin{equation}
  \label{defi_ect_dend}
  \sE_t = \left( (1-\Sigma \sL) \vee (1-\Sigma \sR)\right) \diam ( \sE^c_t , \sE^c ).
\end{equation}
Letting $t=0$, one gets
\begin{equation}
  \label{defi_ec}
    \sE = \left( (1-\Sigma \sL) \vee (1-\Sigma \sR) \right) \circ \sE^c.
\end{equation}

\begin{proposition}
  \label{formule_Ect}
  The series $\sE^c_t$ admits the following expression
  \begin{multline}
    \label{f_ect}
    \frac{1}{1-t}\dun  + \left(\frac{t}{1-t}+b\right) 1 \vee \dN_t\dT\dP - \left(\frac{t}{1-t}+b\right) \dN\dT\dP_t \vee 1 \\ - t \left(\frac{t}{1-t}+b\right) \dN\dT\dP_t \vee \dN_t\dT\dP.
  \end{multline}
\end{proposition}

\begin{proof}
  Using lemma \ref{inverse_cool}, one can invert the relation \eqref{defi_ec} between $\sE$ and $\sE^c$ as
  \begin{equation}
    \label{ec_from_e}
    ((1+\sL)\vee (1+\sR)) \circ \sE = \sE^c.
  \end{equation}

  On the other hand, using lemma \ref{diamant_vee}, lemma
  \ref{double_inversion} and lemma \ref{L_inverse_R}, one can invert
  the relation \eqref{defi_ect_dend} between $\sE_t$ and $\sE^c_t$ to
  obtain
  \begin{equation*}
    \sE_t^c = (1-\Sigma \sR \circ \sE^c) \vee_{\sE_t} (1-\Sigma \sL\circ \sE^c).
  \end{equation*}
  But by \eqref{ec_from_e} and Lemma \ref{nice_compo}, one deduces that
  \begin{equation*}
    \Sigma\sR\circ \sE^c = \Sigma\sR \circ ((1+\sL)\vee (1+\sR)) \circ \sE = -\sL \circ \sE,
  \end{equation*}
  and by symmetry that
  \begin{equation*}
    \Sigma\sL\circ \sE^c = \sL \circ ((1+\sL)\vee (1+\sR)) \circ \sE = -\sR \circ \sE.
  \end{equation*}

  One therefore gets using prop. \ref{uv_from_et} that
  \begin{equation*}
    \sE_t^c = (1 + \sL \circ \sE) \vee_{\sE_t} (1+ \sR\circ \sE)= (1+\sV) \vee_{\sE_t} (1+\sU).
  \end{equation*}
  By the equation \eqref{defi_et}, one finds
  \begin{equation*}
    \frac{1}{1-t}(1+\sV) \vee_{\sD_t} (1+\sU) +\frac{b}{t} ((1+\sV) \vee_{\sD_t} (1+\sU) - (1+\sV) \vee_{\sD} (1+\sU) ).
  \end{equation*}
  Using then the definition \eqref{defi_dt} of $\sD_t$, one gets
  \begin{multline*}
     \left(\frac{1}{1-t}+\frac{b}{t}\right)(1+\sV)*(1+\sU_t) \vee (1+\sV_t)*(1+\sU)-\frac{b}{t}(1+\sV)*(1+\sU) \vee (1+\sV)*(1+\sU) .
  \end{multline*}
  By Lemma \ref{lemmeVU}, one gets
  \begin{equation*}
    \left(\frac{1}{1-t} +\frac{b}{t}\right) (1-t \dN\dT\dP_t ) \vee (1+t \dN_t\dT\dP )-\frac{b}{t} 1 \vee 1.
  \end{equation*}
  This gives the expected result, after simplification.
\end{proof}

Recall from \eqref{from_EC_to_F} that one can write
\begin{equation}
  \sE^c = \dun + b \sF/\dun - b \dun \backslash \sF,
\end{equation}
where $\sF$ is the dendriform image of the series $\xF$ introduced in
\eqref{from_EC_to_F}.

\begin{corollary}
  The series $\sE^c$ admits the following expression
  \begin{equation}
    \dun + b 1 \vee \dN\dT\dP - b \dN\dT\dP \vee 1.
  \end{equation}
  The series $\sF$ is given by
  \begin{equation}
    \label{formule_f_dend}
    - \dN\dT\dP,
  \end{equation}
  and belongs to the descent algebra.
\end{corollary}

One can use \eqref{formule_f_dend} to give an explicit
description of the coefficients of $\sF$.

More precisely, let $\tau$ be a planar binary tree. Then consider the
full canopy of $\tau$, and its coarsest decomposition into blocks of
the shape $\mm^k$ or $\pp^\ell$ for $k,\ell \geq 1$. To each such
block of size $\ell$, one associate a factor $\ca_{\ell}$. 

Then the coefficient of $\tau$ is the product of these factors, times
a power of $b$ given by the number of $\mm$ blocks minus $1$ and times
$(-1)$ to the power the number of $\mm$ minus $1$.

For example, the coefficient of the leftmost planar binary tree of
figure \ref{fig:expl_bt}, whose full canopy is
$\mm\mm\pp\pp\mm\pp\pp$, is
\begin{equation*}
  b \ca_2 \ca_2 \ca_1 \ca_2.
\end{equation*}

Using Prop. \ref{formule_Ect}, one can also give a description of the
coefficients of the series $\sE^c_t$ of connected flows. 

Let us consider a planar binary tree $\tau$ with at least $2$ inner
vertices. One considers the full canopy of $\tau$, and cut it into two
parts by using the position of the root. One distinguish three cases:
the root can either be placed between $\mm$ and $\mm$ at the left of
of the full canopy, or between $\pp$ and $\pp$ at the right of the
full canopy, or between $\pp$ and $\mm$ inside the full canopy. In
each case, one can translate the corresponding term in \eqref{f_ect}
into a description of the factors of the coefficient of $\tau$ in
$\sE^c_t$.

\medskip

The series $\sF$ provides new Lie idempotents.

\begin{proposition}
  \label{lie_id_F}
  The homogeneous component $\sF_n$ of the series $\sF$ satisfies
  \begin{equation}
    \sF_n \cdot \sF_n = n \ca_n \sF_n,
  \end{equation}
  in the symmetric group ring of $\sym_n$.
\end{proposition}

The constant $n\ca_n$ is determined by Lemma \ref{coeff_id}, using
that the coefficient of the ribbon Schur function $\pp^{n-1}$
(corresponding to the full canopy $\mm \pp^n$) is $\ca_n$.


Let us consider now the series $\sF_t=- (1-t) \dN_t \dT \dP_t$, which gives
back $\sF$ when $t=0$. 

Assuming that Conjecture \ref{conjecture_F} holds, one can introduce a
global series $\xF_t$ and propose the following conjecture.

\begin{conjecture}
  The series $\sF_t$ is the dendriform image of the series $\xF_t$.
\end{conjecture}

If this is true, then the series $\sF_t$ provides new Lie idempotents.

\begin{conjecture}
  \label{lie_id_FT}
  The homogeneous component $\sF_{n,t}$ of the series $\sF_t$ satisfies
  \begin{equation}
    \sF_{n,t} \cdot \sF_{n,t} =  n \ca_{n,t} \sF_{n,t},
  \end{equation}
  in the symmetric group ring of $\sym_n$.
\end{conjecture}

The constant $n\ca_{n,t}$ in this conjecture is given by Lemma
\ref{coeff_id}, using that the coefficient of the ribbon Schur
function $\pp^{n-1}$ (corresponding to the full canopy $\mm \pp^n$) is
$\ca_{n,t}$.

Conjecture \ref{lie_id_FT} has been checked up to $\sym_6$ included.

\subsection{Description of $\sZ$}

Let $\sZ$ be the dendriform image of $\xZ$, introduced in
\eqref{Z_here}. We propose here a conjectural description of the
coefficients of $\sZ$.

For positive integers $ p$ and $ q$, let us define polynomials
\begin{equation*}
  z_{p,q}=\sum_{k\geq 0} \binom{p}{k}\binom{q}{k} b^{p+q+1-k}.
\end{equation*}
If $p=q$, this polynomial is essentially a Narayana polynomial of type
$B$.

Let now $\tau$ be a planar binary tree of size $n$. Consider the full
canopy of $\tau$ and decompose it into blocks of the shape
$\mm^{p}\pp^{q}$ with $p,q\geq 1$. To each such block
$\mm^{p}\pp^{q}$, one associates a factor $(-1)^{p-1} z_{p-1,q-1}$.

The coefficient of $\tau$ in the series $\xZ$ seems to be the product
of these factors associated with blocks, divided by $b$. The total
degree with respect to $b$ is $n$ minus the number of blocks.

If this description holds, the coefficient of $\tau$ would depend only
on its canopy. This would imply the following result.

\begin{conjecture}
  The homogeneous component $\sZ_n$ of the series $\sZ$ is in the
  descent algebra and satisfies
  \begin{equation}
    \sZ_n \cdot \sZ_n =  n b^{n-1} \sZ_n,
  \end{equation}
  in the symmetric group ring of $\sym_n$.
\end{conjecture}

Note that one uses Lemma \ref{coeff_id} to get this precise statement.



\begin{question}
  Are the zeroes of the polynomials $z_{p,q}$ real ?
\end{question}

It is known that the generalized Narayana numbers associated with
finite Coxeter groups have only real roots, see \cite[\S
5.2]{reiner_welker}.

\begin{remark}
  The polynomials $z_{p,p+1}$ and $z_{p,p}$, as well as the
  polynomials $\xF$ for forks seem to appear in the article
  \cite{lassalle}, which deals with symmetric functions. The
  relationship with the present work is not clear to us.
\end{remark}

\appendix
\section{Appendix}
\label{appA}

We present here a general setting for the combinatorial use of some
algebraic structures related to operads with specific properties. The
reader may like to keep in mind that the Pre-Lie and dendriform
operads are the motivating examples.

Recall that a species is a functor from the category of finite sets
and bijections to the category of finite sets. For more on the notion
of species, the reader may want to consult the book \cite{belale}.

\subsection{Rooted-operads}

Let $\PP$ be a species, such that $\ZZ\PP$ is endowed with an operad
structure (in the category of $\ZZ$-modules).

We assume that $\PP$ comes with a morphism of species to the species
of pointed sets. This means that to every $\PP$-structure $T$ on a
finite set $I$, one associates an element of $I$. We will call this
element the \textbf{root} of $T$.

We assume also that the composition $\circ_i$ is compatible with the
root in the following sense:
\begin{itemize}
\item if $i$ is the root of $S$,
the root of every term of $S \circ_i T$ is the root of $T$,
\item otherwise
the root of every term of $S \circ_i T$ is the root of $S$.
\end{itemize}

We will call this structure a \textbf{rooted-operad}.

Examples of this situation are provided by the Pre-Lie operad, the
Dendriform operad, the NAP operad \cite{livernet_nap} and the Perm
operad. For the Pre-Lie and NAP operad, the underlying species is the
species of rooted trees, and one takes the root of each tree. For the
Dendriform operad, the underlying species can be described as (rooted)
planar binary trees with labels on internal vertices, and one also
takes the root of each tree. For Perm, the underlying species is the
species of pointed sets, and the root morphism is the identity.

Let $\PP$ and $\PP'$ be two rooted-operads. A \textbf{morphism of
  rooted-operads} $\theta$ from $\PP$ to $\PP'$ is a morphism of
operads from $\ZZ\PP$ to $\ZZ\PP'$ such that for every element $p$ of
$\PP$, the root of every term of $\theta(p)$ is the root of $p$.

\begin{lemma}
  \label{check_on_gen}
  Let $\theta$ be a morphism of operads from $\ZZ\PP$ to $\ZZ\PP'$
  given by its value on elements of $\PP$ that are generators of
  $\ZZ\PP$. If, for every generator $p$, the root of every term in
  $\theta(p)$ is the root of $p$, then $\theta$ is a morphism of
  rooted-operads.
\end{lemma}
\begin{proof}
  Let us prove by induction on the arity of $p\in\PP$ that every term
  in $\theta(p)$ has the same root as $p$. This is clearly true for
  the unit of $\ZZ\PP$.

  Let $x$ be an element of $\PP$. If $x$ is a generator, then the
  statement is true by hypothesis.

  Otherwise, $x$ can be written as a linear combination
  \begin{equation*}
    x = \sum_\alpha \lambda_\alpha p_\alpha \circ_{i_\alpha} q_\alpha,
  \end{equation*}
  where $p_\alpha$ and $q_\alpha$ are elements of $\PP$ of smaller
  arities. Because $\PP$ is a rooted-operad, every composition
  $p_\alpha \circ_{i_\alpha} q_\alpha$ is a linear composition of
  terms sharing the same root. One can therefore remove in the sum
  above every $\alpha$ such that $p_\alpha \circ_{i_\alpha} q_\alpha$
  does not have the same root as $x$. Then one has
  \begin{equation}
    \theta(x)=\sum_\alpha \lambda_\alpha \theta(p_\alpha) \circ_{i_\alpha} \theta(q_\alpha),
  \end{equation}
  and every term $\theta(p_\alpha) \circ_{i_\alpha} \theta(q_\alpha)$
  has the same root as $x$. The induction step is done.
\end{proof}

This condition holds for the morphism $\phi$ from the Pre-Lie operad
to the Dendriform operad, which is therefore a morphism of
rooted-operads.

\subsection{Triple operation associated with rooted-operads}

Let $\PP$ be a rooted-operad, as defined in the previous section.

Let us associate with $\PP$ the coinvariant space $\ZZ\PP_{\sym}$,
which is the free module over $\ZZ$ with basis indexed by isomorphism
classes of $\PP$-structures on all finite sets. One can also describe it as
\begin{equation}
  \ZZ\PP_{\sym} = \oplus_{m \geq 1} \ZZ\PP(m)_{\sym},
\end{equation}
where $\PP(m)$ is $\PP(\{1,2,\dots,m\})$. We will call $[T]$ the basis
element associated with the isomorphism class of the $\PP$-structure
$T$.

On $\ZZ\PP_{\sym}$, there is a natural structure of monoid, whose
associative product is defined using the composition of the operad
$\PP$, see for example \cite[App. A]{qidempotent}. Let us recall this
construction and introduce a refinement of it, which uses the
existence of the root.

Let $s=\sum_m s_m,t=\sum_m t_m$ be elements of $\ZZ\PP_{\sym}$, with
$s_m,t_m$ elements of $\ZZ\PP(m)_{\sym}$. Choose any representatives
$x_m,y_m$ of $s_m,t_m$ in $\ZZ\PP(m)$.

The monoid structure is given by
\begin{equation}
  s \circ  t = \sum_{m \geq 1} \sum_{n_1,\dots,n_{m} \geq 1} \langle \compo{x_m}{y_{n_1},\dots, y_{n_m}} \rangle,
\end{equation}
where $\langle \, \rangle$ is the quotient map to the coinvariant
space, and $(x,y_1,\dots,y_k) \mapsto \compo{x}{y_1,\dots,y_k}$ is the
global composition map of the operad $\PP$, here using the numbering
of the parts to match inputs of $x$ with the $y$'s.

Let now $s=\sum_m s_m,t=\sum_m t_m$ and $u=\sum_m u_m$ be elements of
$\ZZ\PP_{\sym}$, with $s_m,t_m,u_m$ elements of
$\ZZ\PP(m)_{\sym}$. Choose representatives $x_m,y_m$ and $z_m$ of
$s_m,t_m,u_m$ in $\ZZ\PP(m)$, such that the root of $x_m$ is $1$.

Let us introduce the following operation
\begin{equation}
  \label{def_formula_diam}
  s \diam (t, u) = \sum_{m \geq 1} \sum_{n_1,\dots,n_{m} \geq 1} \langle \compo{x_m}{y_{n_1},z_{n_2},\dots, z_{n_m}} \rangle.
\end{equation}
This is well defined, because $\langle
\compo{x_m}{y_{n_1},z_{n_2},\dots, z_{n_m}} \rangle$ does not depend
on the chosen representatives, provided that the root of $x_m$ is $1$.

\begin{proposition}
  The operation $(s,t,u) \mapsto s \diam (t,u)$ satisfies 
  \begin{equation}
    \label{radixe_axiom}
    ( s \diam (t,u) ) \diam (v,w)= s \diam (t \diam (v,w), u \circ  w),
  \end{equation}
  and is linear with respect to $s$ and to $t$. When $t=u$, it reduces
  to the monoid structure:
  \begin{equation}
    s \diam (t,t) = s \circ  t.
  \end{equation}
  The unit $1$ of the operad gives a unit $[1]$, \textit{i.e.} one has
  \begin{equation}
    [1] \diam (t,u) = t \quad \text{and}\quad s \diam ([1],[1])=s.
  \end{equation}
\end{proposition}
\begin{proof}
  The linearity with respect to the parameters $s$ and $t$, the
  special case when $t=u$ and the unit properties all follows by
  inspection from the definition \eqref{def_formula_diam}.

  Concerning formula \eqref{radixe_axiom}, one has to compute both
  sides by choosing representatives with care. It is necessary to
  choose the representatives for $s$ and for $t$ such that their root
  is $1$. Then the result follows from the usual axioms of operads,
  and from the conditions on roots imposed by the definition of
  rooted-operads.
\end{proof}

One could call this kind of structure a \textbf{rooted-monoid}. There
is an obvious notion of morphism of rooted-monoids.

\begin{proposition}
  The construction that maps a rooted-operad $\PP$ to the space
  $\ZZ\PP_\sym$ endowed with the operation $\diam$ is a functor from
  the category of rooted-operads to the category of rooted-monoids.
\end{proposition}
\begin{proof}
  By the definition of the morphisms of rooted operads, the image of
  an element with root $1$ is a sum of terms with root $1$. The functoriality then follows by inspection of the definition \eqref{def_formula_diam}.
\end{proof}

\subsection{Combinatorial use: rooted case}

Let us now present the combinatorial application of the operation
$\diam$ that is used in the main part of the article.

Let $\PP$ be a species, such that $\NN\PP$ is endowed with an
rooted-operad structure (in the category of $\NN$-modules rather than
$\ZZ$-modules).

Examples of this situation are also provided by the Pre-Lie operad,
the Dendriform operad, the NAP operad and the Perm operad.

Let $X$ be a species with a morphism of species to $\PP$. When an
$X$-structure $\alpha$ has image the $\PP$-structure $S$, we will say
that $\alpha$ is over $S$. The set of these structures will be denoted
by $X_{/S}$. Its cardinality only depends on the isomorphism class of
$S$. Let $s_X$ be the generating series
\begin{equation*}
  s_X = \sum_{[S]} \# X_{/S} \frac{[S]}{\aut S},
\end{equation*}
where the sum runs over the set of isomorphism classes of
$\PP$-structures, and $\aut(S)$ is the cardinal of the automorphism
group of any representative $S$ of $[S]$. This is an element of
$\QQ\PP_{\sym}$.

Let $I$ be a finite set. Let $U$ be a $\PP$-structure on $I$. Let
$\pi$ be a partition of $I$. Let $S$ be a $\PP$-structure on the set
of parts of $\pi$, and let $(T_e)_e$ be $\PP$-structures on the parts
$e$ of $\pi$. One will denote by $\mathbf{f}_{S,(T_e)_e}^U$ the
coefficient of $U$ in the global composition $\compo{S}{(T_e)_e}$ in
the operad $\PP$. This is a positive integer by the assumption that
$\NN\PP$ is an operad.

Let us consider now four species $A$, $B$, $C$ and $D$, each one with a
morphism of species to $\PP$.

Suppose that (\textbf{hypothesis} $H_\natural(A,B,C,D)$) for any finite set $I$,
any $r\in I$ and any $\PP$-structure $U$ on $I$ with root $r$, there
is a bijection between
\begin{itemize}
\item the set $D_{/U}$ of $D$-structures over $U$,
\item the set of tuples
  \begin{equation}
    (\pi,S,(T_e)_e,\alpha,(\beta_e)_e,\lambda),
  \end{equation}
  where $\pi$ is a partition of $I$ with a part $\eps$ containing $r$,
  $S$ is a $\PP$-structure with root $\eps$ on the set of parts of
  $\pi$, $(T_e)_e$ are $\PP$-structures on the parts $e$ of $\pi$ such
  that the root of $T_{\eps}$ is $r$, $\alpha\in A_{/S}$,
  $\beta_{\eps} \in B_{/T_{\eps}}$, $\beta_e \in C_{/T_e}$ for
  $e\not=\eps$ and $\lambda \in \{1,\dots,\mathbf{f}_{S,(T_e)_e}^U\}$.
\end{itemize}

\begin{proposition}
  \label{diam_prop}
  Under the hypothesis $H_\natural(A,B,C,D)$, one has
  \begin{equation}
    s_A \diam (s_B, s_C) = s_D.
  \end{equation}
\end{proposition}

\begin{proof}
  Let us fix an integer $m \geq 1$ and an integer $r \in \{1,\dots,m\}$.

  From the hypothesis $H_\natural(A,B,C,D)$, one can obtain the following
  equality
  \begin{equation*}
    \sum \limits_{\underset{\rt(U)=r}{U\in\PP(m)}}\# D_{/U} U = \sum_{\underset{ r\in \eps \in \pi}{\pi \in \prt(m)}} \sum_{S \in \PP(\pi)} \sum_{\underset{ \rt(T_{\eps})=r}{T_e \in \PP(e)}} \#A_{/S} \#B_{/T_{\eps}} \prod_{e\not=\eps} \# C_{/T_e} \compo{S}{(T_e)_e},
  \end{equation*}
  where $\prt(m)$ is the set of partitions of $\{1,\dots,m\}$.

  Let us take the image of this equality by the projection to coinvariants.

  The left hand side becomes
  \begin{equation}
    \label{eq_pour_D}
    \sum_{[U]\in\PP(m)_{\sym}} \frac{(m-1)!\# D_{/U}}{\aut U} [U].
  \end{equation}

  The right hand side becomes
  \begin{equation*}
    \sum_{k \geq 1} \sum_{\underset{ r\in \eps \in \pi}{\pi \in \prt(m,k)}} \sum_{\underset{[T_e] \in \PP(e)_{\sym}}{[S] \in \PP(\pi)_{\sym}}}  \frac{(k-1)!\#A_{/S}}{\aut S}  \frac{(\#\eps-1)!\#B_{/T_{\eps}}}{\aut T_\eps} \prod_{e\not=\eps}   \frac{\#e!\# C_{/T_e}}{\aut T_e}  \langle \compo{S}{(T_e)_e} \rangle,
  \end{equation*}
  where $\prt(m,k)$ is the set of partitions of $\{1,\dots,m\}$ into
  $k$ parts.

  By using the factor $(k-1)!$ to define an order on the parts of the
  partition, such that the first part contains the root, one gets
  \begin{equation*}
    \sum_{k \geq 1} \sum_{\underset{ r\in \pi_1}{\pi_1,\dots,\pi_k}} \sum_{\underset{[T_i] \in \PP(\pi_i)_{\sym}}{[S] \in \PP(k)_{\sym}}}  \frac{\#A_{/S}}{\aut S}  \frac{(\#\pi_1-1)!\#B_{/T_1}}{\aut T_1} \prod_{j=2}^{k}   \frac{\#\pi_j!\# C_{/T_j}}{\aut T_j}  \langle \compo{S}{(T_j)_j} \rangle.
  \end{equation*}

  By using the multinomial formula for the number of ordered
  partitions of $m$ into $k$ parts of size $n_1,\dots,n_k$, one gets
  \begin{equation}
    \label{eq_pour_ABC}
    (m-1)! \sum_{k \geq 1} \sum_{\underset{n_1+\dots+n_k=m}{n_1,\dots,n_k \geq 1}} \sum_{\underset{[T_i] \in \PP(n_i)_{\sym}}{[S] \in \PP(k)_{\sym}}}   \frac{\#A_{/S}}{\aut S}  \frac{\#B_{/T_1}}{\aut T_1} \prod_{j=2}^k  \frac{\# C_{/T_j}}{\aut T_j}  \langle \compo{S}{T_1,(T_j)_{j\geq 2}} \rangle.
  \end{equation}

  From the equality between \eqref{eq_pour_D} and \eqref{eq_pour_ABC},
  one deduces the result, after division by $(m-1)!$ and summation
  over $m \geq 1$.
\end{proof}

If $B=C$, the operation $\diam$ reduces to the monoid product $\circ
$, and the hypothesis $H_\natural(A,B,B,D)$ can be formulated without
using the root. One can obtain in this way a simpler analog of
proposition \ref{diam_prop} valid for any operad on a species
$\PP$. This is made explicit in the next paragraph.

\subsection{Combinatorial use: group case}

Let $\PP$ be a species, such that $\NN\PP$ is endowed with an operad
structure.

Let us consider now three species $A$, $B$ and $C$, each one with a
morphism of species to $\PP$.

Suppose that (\textbf{hypothesis} $H_\sharp(A,B,C)$) for any finite
set $I$ and any $\PP$-structure $U$ on $I$, there is a bijection
between
\begin{itemize}
\item the set $C_{/U}$ of $C$-structures over $U$,
\item the set of tuples
  \begin{equation}
    (\pi,S,(T_e)_e,\alpha,(\beta_e)_e,\lambda),
  \end{equation}
  where $\pi$ is a partition of $I$, $S$ is a $\PP$-structure on the
  set of parts of $\pi$, $(T_e)_e$ are $\PP$-structures on the parts
  $e$ of $\pi$, $\alpha\in A_{/S}$, $\beta_{e} \in B_{/T_e}$ and $\lambda \in
  \{1,\dots,\mathbf{f}_{S,(T_e)_e}^U\}$.
\end{itemize}

\begin{proposition}
  \label{circ_prop}
  Under the hypothesis $H_\sharp(A,B,C)$, one has
  \begin{equation}
    s_A \circ s_B = s_c.
  \end{equation}
\end{proposition}

\bibliographystyle{plain}
\bibliography{idem_nara}

\end{document}